\documentclass[a4paper,12pt]{amsart}
\usepackage[utf8]{inputenc}
\usepackage{amssymb,amsmath}
\usepackage{amscd}
\usepackage[all]{xy}
\usepackage[onehalfspacing]{setspace}
\usepackage{enumerate}
\usepackage{graphicx}
\usepackage{url}
\usepackage[a4paper, left=2cm, right=2cm, top=3.5cm, bottom=3cm]{geometry}
\usepackage{tikz-cd}
\usepackage{array}
\usepackage[normalem]{ulem}
\usepackage{verbatim}
\usepackage[colorlinks, linktocpage, citecolor = blue, linkcolor = blue,backref]{hyperref} 

\def\A{\mathbb A}
\def\C{\mathbb C}
\def\F{\mathbb F}

\def\L{\mathbb L}
\def\bL{\mathbb L}
\def\P{\mathbb P}
\def\bP{\mathbb P}
\def\R{\mathbb R}
\def\bQ{\mathbb Q}
\def\Q{\mathbb Q}

\def\Z{\mathbb Z}
\def\Gset{\mathbf{Gset}}
\def\EtSch{\mathbf{EtSch}}

\def\CC{\mathcal C}

\def\HH{\mathcal H}

\def\MM{\mathcal M}
\def\cM{\mathcal M}

\def\OO{\mathcal O}

\def\chr{{\operatorname{char}}}

\def\Ker{{\operatorname{Ker}}}

\def\Spec{{\operatorname{Spec}}}

\def\NS{{\operatorname{NS}}}

\def\Bl{{\operatorname{Bl}}}
\def\bl{{\operatorname{bl}}}

\def\Aut{{\operatorname{Aut}}}

\def\rk{{\operatorname{rk}}}

\def\chr{{\operatorname{char}}}
\def\Bir{{\operatorname{Bir}}}
\def\Cr{{\operatorname{Cr}}}
\def\Rep{{\operatorname{Rep}}}
\def\Burn{{\operatorname{Burn}}}

\def\Gal{{\operatorname{Gal}}}

\def\wt{\widetilde}

\def\ol{\overline}

\def\Kt{{\mathrm{K}}}

\def\GL{{\mathrm{GL}}}

\def\et{{\text{\'et}}}

\def\Var{{\mathrm{Var}}}

\theoremstyle{plain}
\newtheorem{dummy}{dummy}[section]
\newtheorem{theorem}[dummy]{Theorem}
\newtheorem{proposition}[dummy]{Proposition}

\newtheorem{lemma}[dummy]{Lemma}
\newtheorem{lem}[dummy]{Lemma}
\newtheorem{lemdef}[dummy]{Lemma-Definition}
\newtheorem{corollary}[dummy]{Corollary}
\newtheorem{cor}[dummy]{Corollary}
\newtheorem{claim}[dummy]{Claim}
\newtheorem{example}[dummy]{Example}

\newtheorem{definition}[dummy]{Definition}
\newtheorem{remark}[dummy]{Remark}
\newtheorem{question}[dummy]{Question}
\numberwithin{equation}{section}

\newcommand{\ti}[1]{\tilde{#1}}

\newcommand{\vast}{\bBigg@{4}}
\newcommand{\Vast}{\bBigg@{5}}

\newcommand{\ga}{\alpha}
\newcommand{\gb}{\beta}
\newcommand{\gd}{\delta}

\newcommand{\go}{\omega}

\newcommand{\dto}{\dashrightarrow}
\newcommand{\lto}{\leftarrow}

\newcommand{\hto}{\hookrightarrow}
\newcommand{\xto}[1]{\xrightarrow{ #1 }}
\newcommand{\xlto}[1]{\xleftarrow{ #1 }}

\newcommand{\cnec}{\mathrel{:=}}

\title[Factorization centers in dimension two]
{Factorization centers in dimension two
and the Grothendieck ring of varieties}



\author{Hsueh-Yung Lin, Evgeny Shinder, Susanna Zimmermann}

	\address{Department of Mathematics, National Taiwan University, 
	and National Center for Theoretical Sciences,
	Taipei, Taiwan.}
	\email{hsuehyunglin@ntu.edu.tw}

\address{
School of Mathematics and Statistics, University of Sheffield,
Hounsfield Road, S3 7RH, UK, and
Hausdorff Center for Mathematics
at the University of Bonn, Endenicher Allee 60, 53115.
}
\email{eugene.shinder@gmail.com}

\address{Laboratoire angevin de recherche en math\'ematiques (LAREMA), CNRS, Universit\'e d'Angers, 49045 Angers cedex 1, France}
\email{susanna.zimmermann@universite-paris-saclay.fr}

\begin{document}

\raggedbottom

\maketitle

\begin{abstract}
We initiate the study of factorization centers of birational maps, and complete it for surfaces over a perfect field in this article.
We prove that 
for every birational automorphism $\phi : X \dashrightarrow X$ of a smooth projective surface $X$ over a perfect field $k$,
the blowup centers are isomorphic to the blowdown centers 
in every weak factorization of $\phi$. 
This implies that 
nontrivial L-equivalences of $0$-dimensional varieties 
cannot be constructed 
based on birational automorphisms of a surface. 
It also implies that
rationality centers are well-defined for every rational surface $X$,
namely there exists a $0$-dimensional variety intrinsic to $X$, 
which is blown up in any rationality construction of $X$.\\
\end{abstract}

\tableofcontents 

\section{Introduction}

One source of motivation in birational geometry comes from studying groups of
birational automorphisms of algebraic varieties, in particular
the Cremona groups $\Cr_n(k) = \Bir(\P^n_k)$.
Each birational automorphism blows up
some subschemes
and contracts some 
exceptional divisors.
The primary question we study in this paper is:
\begin{question}\label{q:main}
Let $\phi: X \dashrightarrow X$
be a birational automorphism of a smooth projective variety.
Do centers blown up by $\phi$
correspond, up to stable birational equivalence,
to 
the exceptional divisors blown down by $\phi$?
\end{question}

We note that exceptional divisors are ruled over the corresponding blow up centers, so 
asking about stable
birational equivalence classes 
is a natural way to compare
exceptional divisors
with the blow up centers.
We give a complete answer to 
a stronger version of Question \ref{q:main} for surfaces over an arbitrary perfect field,
see Theorem \ref{thm:main-intro} below.
Our proof is an application
of the 
two-dimensional Minimal Model Program
\cite{IskovskikhMinModel, IskovskikhFact}
combined with the Grothendieck ring
of varieties,
\'etale cohomology groups 
and the so-called Gassmann equivalence of Galois sets.

There are three main reasons to study Question \ref{q:main}.
First of all, it is the structure of birational automorphisms and
Cremona groups, in particular
their generation by involutions or regularizable elements.
In the recent paper \cite{BirMot},
we explain that
the answer to Question
\ref{q:main}
is negative in various contexts
in dimension $n \ge 3$,
and give applications to
the structure
of the higher Cremona groups.
Secondly, Question \ref{q:main} has a tight relationship
to the structure of
the Grothendieck ring of varieties.

\begin{question}\label{q:groth}
\hfill

\begin{enumerate}
    \item[(a)] (Larsen-Lunts \cite{LarsenLunts}, slightly reformulated)\label{q:Larsen-Lunts}
If classes of smooth projective
varieties $[X]$ and $[Y]$
coincide in the Grothendieck ring, 
how can we compare the geometry of $X$ and $Y$?
For instance,
are $X$ and $Y$ birational?

\item[(b)] \cite{KuznetsovShinder}
What is the geometric meaning
of L-equivalence?
For instance,
if zero-dimensional schemes
are L-equivalent,
do they have to be isomorphic?

\item[(c)] \cite{MO-Dominik}
Does the Grothendieck ring of varieties $K_0(\Var/k)$
have torsion elements?

\end{enumerate}
\end{question}

In the direction of (a),
the main result of \cite{LarsenLunts},
which also follows from \cite{Bittner},
is that if $\chr(k) = 0$, equality
of classes $[X]$, $[Y]$ of smooth
projective connected varieties 
implies that $X$ and $Y$ are \emph{stably} 
birational.
On the other hand,
for non-projective smooth connected
varieties, the second part of
Question \ref{q:Larsen-Lunts}.(a)
is known to have a negative
answer (the first such example is~\cite[proof of Theorem 2.13]{Borisov}).
For (b), 
see \cite{KuznetsovShinder}
for conjectural 
relations to derived equivalence.
In \S \ref{subsec:L-eq}, we explain
that L-equivalence of smooth zero-dimensional
schemes implies Gassmann equivalence of the corresponding
Galois sets, but this does not rule out the possibility
of nontrivial L-equivalence between such schemes.
Nothing is known about (c).
As the Grothendieck ring is a colimit of the truncated
groups $K_0(\Var^{\le n}/k)$ generated by varieties of dimension
up to $n$, we can ask each of the questions in these
truncated groups.
As a direct consequence of our
positive answer to Question \ref{q:main}
in dimension two,
we are able to answer
Question \ref{q:groth} completely for $K_0(\Var^{\le 2}/k)$;
see Corollary \ref{cor:main}. Namely,
 (a) 
 equality of classes of smooth projective varieties
in $K_0(\Var^{\le 2}/k)$ 
implies birationality,
(b) L-equivalence in $K_0(\Var^{\le 2}/k)$ is trivial and 
(c) $K_0(\Var^{\le 2}/k)$ is a free abelian group. 
We expect
that studying
Question \ref{q:main}
in dimensions $\le n$ would lead to
good control over Question \ref{q:groth}
for $K_0(\Var^{\le n}/k)$.

Finally, answering Question \ref{q:main}
positively, or explaining all ways in which it can fail,
allows to control all rationality constructions
for every rational variety, see
the discussion of rationality centers
below and
\S \ref{subsec:rat-cent} for more
details.

\medskip

We now explain our answer to Question \ref{q:main} in dimension $2$
over perfect fields.
Our main result
can be stated in the following way.

\begin{theorem}[see Theorem \ref{thm:main}]
\label{thm:main-intro}
Let $k$ be a perfect field.
Let $X/k$ be a smooth projective surface and let $\phi \in \Bir(X)$
be a birational automorphism.
For any factorization of $\phi$
into a sequence of blow ups and blow downs 
at
connected smooth zero-dimensional subschemes,
let $Z_1, \dots, Z_r$ (resp. $Z_1', \dots, Z_{r'}'$) be the 
centers 
which get
blown up (resp. blown down). 
Then $r = r'$ and there is a 
reordering
under which $Z_i \simeq Z_i'$ over $k$ for all $i = 1, \dots, r$. 
\end{theorem}

Note that Theorem~\ref{thm:main-intro} 
is easily seen to hold
when $k$ is algebraically closed field, as the Galois actions
are trivial and the number
of points blown up 
is equal to the number of exceptional
divisors contracted. 
Over an arbitrary perfect field, 
Theorem~\ref{thm:main-intro} is a non-trivial statement. 
In particular, by this we mean that
it is not possible
to recover centers of the blow ups 
simply
from the Galois action
on the cohomology of the surface;
see Remark \ref{rem:GCM} for the technical
formulation of this statement in terms of 
Chow motives, and Example \ref{ex:ex-cubics} for an explicit 
construction.
Neither there seems
to exist a straightforward geometric
argument:
see Example~\ref{ex:dP5-links} for an 
illustration of the birational geometry involved.

To put Theorem \ref{thm:main-intro} into an appropriate context, we
introduce a general
invariant $c(\phi)$, keeping track of
the \emph{factorization centers} of the birational map $\phi$,
which is a homomorphism
from the groupoid of birational types of surfaces to a 
free abelian group generated by reduced $k$-schemes of dimension 0;
see Corollary \ref{cor:c-axioms}
for an axiomatic definition of $c(\phi$).
We then show that $c$ is constant
on each $\Bir(X,Y)$ and as a consequence
$c(\phi) = 0$ for every self-map, which implies
Theorem \ref{thm:main-intro}.

To prove 
that $c$ is constant
on $\Bir(X,Y)$,
we have to consider each birational type of 
surfaces that can occur,
with geometrically rational (and especially, rational)
surfaces being the most interesting ones.
For geometrically rational surfaces, 
by the two-dimensional minimal model program 
we have to consider birational maps between
del Pezzo surfaces and conic bundles.
Our proof for uniqueness of factorization
centers
uses two ingredients: \emph{Sarkisov links}
and \emph{Gassmann equivalence}. 
Sarkisov links are certain elementary
birational transformations between del Pezzo surfaces
and conic bundles which 
generate the groupoid of birational maps between 
geometrically rational surfaces. 
In dimension two, the existence of
decomposition into links
has been proved and all
the links
have been
classified by Iskovskikh into a finite
list \cite{IskovskikhFact}. The largest variety of links
occurs for what we call \emph{models
of large degree}: these include
all minimal geometrically rational
surfaces with $K_X^2 \ge 5$.

We derive Theorem \ref{thm:main-intro}
for geometrically rational surfaces
from a uniform claim we make for all links between minimal
geometrically rational surfaces (see Proposition~\ref{prop:geom-rat-5}), 
which we check for each link in Iskovskikh's classification.
The latter uniform claim is formulated in terms of the (virtual) N\'eron-Severi
Galois set, which is closely related to the N\'eron-Severi lattice of
the surface as a Galois module.
These N\'eron-Severi sets are defined
in terms of the Galois action on
linear systems of
rational curves of low degree (typically pencils
of conics and nets of twisted cubics) on del Pezzo
surfaces.

On top of the classification of links, to prove the result
we use \'etale cohomology, permutation modules, and the so-called
Gassmann triples from group theory: these are triples $(G, H, H')$
with $H$ and $H'$ subgroups of finite group $G$ such that
$\C[G/H]$ and $\C[G/H']$ are isomorphic $G$-representations (this holds if $H$, $H'$
are conjugate, but the converse if false).
Gassmann triples are used to produce
arithmetically equivalent
fields, see e.g. \cite{Perlis, Gassmannleq6},
isospectral manifolds \cite{Sunada},
as well as curves with isomorphic Jacobians \cite{PrasadGassmann}.
In our dealing with Gassmann equivalence,
we follow the approach of \cite{ParzanchevskiGsets} which
generalizes Gassmann triples
from a pair of subgroups of $G$
to a pair of $G$-sets with possibly non-transitive actions.
Our Gassmann equivalent sets come from \'etale cohomology groups and
in particular, from N\'eron-Severi groups of geometrically
rational surfaces
with their Galois group action.
We show that Gassmann equivalence provides
a cohomological expression of L-equivalence,
see Lemma \ref{lem:L-eq-Gass} and Remark~\ref{rem:GCM}.
Using the fact that Gassmann triples of small order
are trivial allows us 
to significantly limit the number of Sarkisov
links we have to consider in the proof of the main result. 

As a by-product of our results on Gassmann equivalence
we obtain conditions 
which forbid
non-trivial L-equivalence of
zero-dimensional schemes to exist;
for example L-equivalent zero-dimensional
connected reduced
schemes (resp. reduced schemes) of degree $\le 6$ (resp. $\le 5$)
are isomorphic and
L-equivalence always implies isomorphism
for fields with procyclic Galois groups
such as
$k = \R$, $k = \F_q$, see Example \ref{ex:R-Fq}.

Finally for rational surfaces, our main result has the following consequence: 
if $X$
is a smooth projective rational surface, then there exists a zero-dimensional scheme depending only on $X$, which will have to be blown up
by any birational isomorphism $\phi: \P^2 \dashrightarrow X$,
see Corollary \ref{cor:rat-centers}; we think
of these associated schemes as \emph{rationality centers}
of $X$.
This is in contrast with the higher-dimensional geometry, where
the associated rationality centers
are not well-defined, even up to stable
birational equivalence. For instance,
a K3 surface associated to a cubic fourfolds should not
be unique up to
isomorphism: 
it should be unique
up to derived equivalence \cite[Remark 27]{Hassett-rationality}, 
or possibly up to L-equivalence, as hinted
in \cite[(2.6.1)]{KuznetsovShinder}.

This text is organized as follows. \S\ref{sec:Gassmann} is devoted to Gassmann equivalence of $G$-sets.  We relate it to L-equivalence and provide sufficient conditions for two Gassmann equivalent $G$-sets to be isomorphic. In \S\ref{sec:main}, we define the invariant $c(\phi)$ which captures the factorization center of a birational map $\phi$ and formulate the main theorem (Theorem \ref{thm:main}) of the paper, which we prove in \S\ref{sec:birsurf} and \S\ref{sec:geom-rat-5}. We also study rational curves on del Pezzo surfaces in \S\ref{sec:birsurf}, which will be used to define and study the virtual Néron-Severi sets in \S\ref{sec:geom-rat-5}. 
At the end of the paper we explain the concept of
rationality center
for rational surfaces, and illustrate it in the case of del Pezzo surfaces.

\subsection*{Acknowledgements}

H.Y.L. was supported by the World Premier International Research centre Initiative (WPI), MEXT, Japan,
then by the Ministry of Education Yushan Young Scholar Fellowship (NTU-110VV006) and the National Science and Technology Council (110-2628-M-002-006-).
E.S. is supported by
EP/T019379/1 ``Derived categories and algebraic K-theory of singularities'', and by the
    ERC Synergy grant ``Modern Aspects of Geometry: Categories, Cycles and Cohomology of Hyperkähler Varieties". 
S.Z. was supported by the ANR Project FIBALGA ANR-18-CE40-0003-01 and the Project \'Etoiles montantes of the R\'egion Pays de la Loire.
E.S.
would like to thank 
Yujiro Kawamata, Keiji Oguiso,
Atsushi Takahashi 
and Shinnosuke Okawa
for supporting
his visit to Japan 
where
some of this work has originated.
E.S. would also like to thank
Ivan Cheltsov, Sergey Galkin,
Alexander Kuznetsov, 
Yuri Prokhorov,
Constantin Shramov
for discussions and encouragement, 
and Jean-Louis Colliot-Th\'el\`ene, Brendan Hassett
and  Claire Voisin for their comments on this work.
H.Y.L. would like to thank the NCTS in Taipei
for the hospitality and support during the preparation of this paper.
S.Z. would like to thank Jean-Louis Colliot-Th\'el\`ene for his comments on this work and Jean-Louis Colliot-Th\'el\`ene,  Alexander Merkurjev and Jean-Pierre Serre for indicating references to examples of integral Gassmann triples. 
The authors would like to thank Keiji Oguiso for the initial discussion he involved on this topic and for making this collaboration possible.

\subsection*{Conventions}

We work over a perfect field $k$ unless otherwise specified.
All schemes are of finite type over $k$.
By a surface we mean a connected
smooth projective
(but not necessarily geometrically irreducible)
surface $X$ over $k$.

\section{Gassmann equivalence}\label{sec:Gassmann}

We explain the basics about the Burnside
ring of a profinite 
group $G$, the Grothendieck ring
of varieties and a homomorphism between them
when
$G$ is the absolute Galois
group.
Our presentation
is similar to \cite{Rokaeus}\footnote{
Note however that \cite{Rokaeus}
makes
an erroneous
statement on p. 943 
that the homomorphism
from the Burnside ring
to the representation ring is injective
on transitive $G$-sets (\textit{cf.} Example
\ref{ex:gassmann}(2)).}.
This homomorphism allows us to relate
L-equivalence of reduced
zero-dimensional schemes
to Gassmann equivalence of $G$-sets,
and thus to rule out the possibility of nontrivial
L-equivalence in small degree.

\subsection{Definition and basic properties}
\hfill

Let $G$ be a profinite group and let $\Gset$ be the semi-ring of  isomorphism classes of finite $G$-sets  on which $G$ acts continuously for the profinite topology, where
finite sets are considered with discrete topology. Continuity of the action is equivalent
to the requirement that stabilizer of any point
is open (and in particular, a finite index subgroup). Here in $\Gset$, the addition (resp. multiplication) is defined by 
disjoint unions (resp. Cartesian products). We define the Burnside ring 
$\Burn(G)$ of $G$ to be the Grothendieck ring associated to $\Gset$. When 
$G$ is a finite group, the definition of $\Burn(G)$ is the classical one. 
We sometimes refer to elements of $\Burn(G)$,
that is, combinations of isomorphism classes of $G$-sets
with integer coefficients, as \emph{virtual $G$-sets}.
We note that the number of elements 
(resp. the number of orbits)
in a $G$-set
gives rise to a ring homomorphism (resp. a group homomorphism) 
$\Burn(G) \to \Z$.

Let $F$ be a field of characteristic zero and let $\Rep(G,F)$ be the abelian monoidal category of finite dimensional $G$-representations over $F$. 
Let $K_0(\Rep(G,F))$
be the Grothendieck ring of $\Rep(G,F)$. 
There is a well-defined ring homomorphism
\begin{equation}\label{eq:mu}
\mu_G: \Burn(G) \to K_0(\Rep(G, F))
\end{equation}
which sends the class of a continuous finite $G$-set $A$ to the class
of the permutation representation $F[A]$.
We are interested in the kernel of this homomorphism.

\begin{lemdef}\label{lem-equivGm}
Let $G$ be a profinite group and let $A$, $B$ be continuous finite $G$-sets. Fix a field $F$ of characteristic zero. The following conditions are equivalent:
\begin{enumerate}
    \item $\mu_G(A) = \mu_G(B)$ in $K_0(\Rep(G, F))$. 
    \item $F[A] \simeq F[B]$ as $F[G]$-modules.
\end{enumerate}
We say that two continuous finite $G$-sets $A$ and $B$ are \emph{Gassmann equivalent} if they satisfy one of the above equivalent conditions.
\end{lemdef}

\begin{proof}
It is clear that (2) implies (1). Conversely, if $\mu_G(A) = \mu_G(B)$, then since $F[A]$ and $F[B]$ are finite dimensional (as $F$-vector spaces), both $F[A]$ and $F[B]$ admit composition series in $\Rep(G, F)$ of finite length~\cite[Lemma 3.9]{ErdmannHolmRep}. The Jordan-H\"older theorem~\cite[Excercice II.6.3]{Weibel-Kbook} shows that $F[A]$ and $F[B]$ have the same collection of isomorphism classes of (simple) Jordan-H\"older factors. As both $G$-modules $F[A]$ and $F[B]$ factor through a finite quotient $G/H$ of $G$ (we can take $H$ to be the intersection of
the kernels of the two actions, which is an open subgroup of $G$,
hence has finite index), Maschke's Theorem implies that $F[A]$ and $F[B]$ are semisimple $G$-modules. Hence $F[A] \simeq F[B]$.
\end{proof}

The following lemma reduces
Gassmann equivalence to the case
of finite group actions.
If $A$ is a $G$-set,
we write $G^A$ for the kernel
of the action; thus $G^A$ is a normal
subgroup
and $G/G^A$ acts on $A$ faithfully.

\begin{lem}\label{lem-redft}
Two  continuous finite $G$-sets
 $A$ and $B$ are Gassmann equivalent 
 if and only if $G^A = G^B =: H$ and $A$ and $B$ are Gassmann equivalent as $G/H$-sets.
\end{lem}

\begin{proof}
This follows from Lemma~\ref{lem-equivGm}, together with the observation that 
$G^A = \ker{\rho_A}$ (and similarly $G^B = \ker(\rho_B)$) where $\rho_A : G \to \GL(V,F)$ is the homomorphism 
defining the $F[G]$-module structure of $F[A]$ 
and $V$ is the underlying $F$-vector space of $F[A]$.
\end{proof}

\begin{lemma}[{\cite[Proposition 1]{ParzanchevskiGsets}, 
cf \cite[Definition 1]{PrasadGassmann}}]
\label{lem:def-Gassmann}
Let $G$ be a finite group and let $A$ and $B$ be finite $G$-sets. Fix a field $F$ of characteristic zero.
The following conditions are equivalent:
\begin{enumerate}
    \item For every $g \in G$, $g$ fixes the same
    number of elements in $A$ and in $B$.
    \item There exist subgroups $H_1, \dots, H_r$,
    $H_1', \dots, H_r'$ of $G$
    such that 
    \[
    A \simeq \bigsqcup_{i=1}^r G/H_i, \;\;\;\;
    B \simeq \bigsqcup_{i=1}^r G/H_i'
    \]
    and for each
    conjugacy class $T \subset G$
    these subgroups
    satisfy
    \begin{equation}\label{eq:Hi}
    \sum_{i=1}^r \frac{|T \cap H_i|}{|H_i|} = \sum_{i=1}^r \frac{|T \cap H_i'|}{|H_i'|}.
    \end{equation}
    \item $A$ and $B$ are Gassmann equivalent.
\end{enumerate}
In particular, the kernel of $\mu$ is independent
of the choice of $F$.
\end{lemma}
\begin{proof}

Equivalence of (1), (2) and (3) for $F = \C$ is proved
in \cite[Proposition 1]{ParzanchevskiGsets}.
Finally, condition (3) is independent of
the choice of $F$ because the functor
\[
\Rep(G, \Q) \to \Rep(G, F) 
\]
is injective on isomorphism classes~\cite[\S 14.6]{MR0450380}. 
\end{proof}

\begin{remark}\label{rem:orbits}
The number of $G$-orbits in Gassmann equivalent continuous $G$-sets
is the same. This follows from Lemma~\ref{lem-redft} and Lemma \ref{lem:def-Gassmann}(2), 
or directly as the number of
orbits equals the multiplicity
of the trivial representation in the permutation representation.
\end{remark}

\begin{proposition}\label{prop:Gassmann}
Let $G$ be a profinite group and let $A$ and $B$ be Gassmann equivalent continuous finite $G$-sets.
If one of the following conditions is satisfied:
\begin{itemize}
    \item $G/G^A$ is a cyclic group
    \item $A$ is transitive and the stabilizer of a point is normal in $G$
    \item $A$ is transitive and $|A| \le 6$
    \item $|A| \le 5$
\end{itemize}
then $A$ and $B$ are isomorphic $G$-sets.
\end{proposition}

If $\alpha \in \Burn(G)$, then
we write $|\alpha|$ for the smallest
$\max(|A|,|B|)$ over all
possible representations $\alpha = [A] - [B]$
with continuous $G$-sets $A$ and $B$.
We have the following immediate corollary of Proposition~\ref{prop:Gassmann}. 

\begin{cor}\label{cor:Gassmann}
Let $\alpha \in \Ker(\mu_G)$.
If 
$|\alpha| \le 5$, then $\alpha = 0$.
\end{cor}

\begin{proof}[Proof of Proposition~\ref{prop:Gassmann}]
By Lemma~\ref{lem-redft}, we can assume that $G$ is a finite group and the $G$-actions on both $A$ and $B$ are faithful.
When $G$ is a finite cyclic group, the result is~\cite[Proposition 4.1]{ParzanchevskiGsets}.
If $A$ is transitive, then by Remark \ref{rem:orbits},
$B$ is also transitive. Let $H$ (resp. $H'$)
be the stabilizer of an element $a \in A$ (resp. $b \in B$).
Then $G/H$ and $G/H'$ are Gassmann equivalent;
in this case one refers to $(G,H,H')$ as a Gassmann triple.
It is well-known and easy to see 
(e.g. using \eqref{eq:Hi} with $r=1$)
that if $H$ is normal, then
$H' = H$. 
It is a nontrivial computation that if $[G:H] \le 6$
in Gassmann triple $(G,H,H')$,
$H$ and $H'$ will be conjugate so that $A$ and $B$ are isomorphic, see \cite[Proof of Theorem 3]{Perlis} or~\cite[p.3]{Gassmannleq6}.

Finally, we need to consider the case when $A$ may not be transitive
but $t:=|A| \le 5$. We write $A = \{1,\dots, t\}$ and let $S_t$ be the permutation group of $t$ elements. Since the $G$-action on $A$ is assumed to be faithful, it realizes $G$ as a subgroup of $S_t$.

Up to adding trivial $G$-sets $\{*\}$, we can assume that $t = 5$. By Remark~\ref{rem:orbits}, $A$ and $B$ have the same number
of $G$-orbits, and we argue according to the number of orbits $m$. When $m = 5$, then both $A$ and $B$ are trivial $G$-sets. When $m = 4$, then both $A$ and $B$ are isomorphic to  $\{1,2\} \sqcup \{*\} \sqcup \{*\} \sqcup \{*\}$ and $G \simeq \Z/2\Z$ acts on $\{1,2\}$ by involution. The case $m=1$ is already covered by the third case of Proposition~\ref{prop:Gassmann}. So it remains to study the case where $ m = 2$ or $3$.

First we show that $G$ uniquely determines the length of the orbits of the $G$-set $A$. 
If $m=3$, then $A$ is isomorphic to either $A_3  \sqcup \{*\} \sqcup \{*\}$ or $A_{2}  \sqcup A'_2  \sqcup \{*\}$ with $|A_3| = 3$ and $|A_2| = |A'_2| = 2$. So $G \le S_3$ in the former case and $G \le \Z/2\Z \times \Z/2\Z$ in the latter case. In the former case, since the $G$-action is transitive on $A_3$,  $G$ contains $\Z/3\Z \le S_3$, so we cannot embed $G$ into $\Z/2\Z \times \Z/2\Z$. Therefore $G$ uniquely determines the length of the orbits. If $m =2$, then $A$ is isomorphic to either $A_3  \sqcup A_2$ or $A_{4}  \sqcup \{*\}$ with $|A_i| = i$. So $G \le S_3 \times \Z/2\Z$ in the former case and $G \le S_4$ in the latter case. In the former case, since the $G$-action is transitive on both $A_3$ and $A_2$, $G$ contains $ \Z/3\Z \times \Z/2\Z \le S_3 \times \Z/2\Z$. Since $S_4$ 
has no elements of order six, we cannot embed $G$ into $S_4$. Hence the lengths of the orbits are determined by $G$ as well.

We still assume that $m = 2$ or $3$. If $A = A' \sqcup A''$ as $G$-sets where $A'$ is transitive and $A''$ is a disjoint union of trivial $G$-sets, then we also have the same type of decomposition $B = B' \sqcup B''$ with $|A'| = |B'|$. It follows that $\mu_G(A') = \mu_G(B')$, so $A' \simeq B'$ since both $G$-sets are transitive. This covers the cases where $A = A_3  \sqcup \{*\} \sqcup \{*\}$ and  $A = A_4  \sqcup \{*\}$. 
If $A = A_{2}  \sqcup A'_2  \sqcup \{*\}$, then either $G = \Z/2\Z \times \Z/2\Z$ or $G$ is the diagonal  $\Z/2\Z \le \Z/2\Z \times \Z/2\Z$. In either case we verify that $G$ uniquely determines the $G$-set structure of $A$. If  $A = A_3  \sqcup A_2$, then either $G \le \Z/3\Z \times \Z/2\Z$ or $G = S_3 \times \Z/2\Z$.
Once again in either case, we verify that $G$ uniquely determines the $G$-set structure of $A$.
\end{proof}

The following example shows that the lower bounds
on the order of $G$-sets in Proposition~\ref{prop:Gassmann} 
are optimal.

\begin{example}\label{ex:gassmann}
(1) \cite[1.1]{ParzanchevskiGsets}
Let $G = \Z/2\Z  \times \Z/2\Z$ 
and let 
\[
H_1 = \langle (1,0) \rangle, \;\;
H_2 = \langle (0,1) \rangle, \;\;
H_3 = \langle (1,1) \rangle,
\]
and
\[
 H_1' = \{(0,0)\}, \;\; H_2' = H_3' = G.
\] 
Then $H_1, H_2, H_3$ and $H_1', H_2', H_3'$
satisfy conditions of Lemma \ref{lem:def-Gassmann}(2)
thus give rise to nonisomorphic nontransitive Gassmann equivalent
sets of order $6$ with orbit decompositions
$2+2+2$ and $4+1+1$ respectively.
Indeed for abelian groups the condition \eqref{eq:Hi} rewrites
as
\[
\sum_{H_i \ni g} \frac1{|H_i|} = \sum_{H'_i \ni g} \frac1{|H'_i|} \text{ for all $g \in G$}
\]
which is immediately verified.

(2) \cite[p. 358]{Perlis} Let $G = \mathrm{PSL}_2(\F_7) \simeq \mathrm{PSL}_3(\F_2)$ be
the simple group of order $168$. Let $A$ be the set of $\F_2$-points
of the projective plane over $\F_2$, 
and $B$ be the set of $\F_2$-lines on this plane, that
is points of the dual projective plane.
Then $A$, $B$ are transitive $G$-sets of order $7$.
Using simple linear algebra of the $\mathrm{PSL}_3(\F_2)$-action on $\P^2_{\F_2}$,
one shows that $A$ and $B$ satisfy
Lemma \ref{lem:def-Gassmann}(1), so that $A$ and $B$ are Gassmann
equivalent, and one can check that they are not isomorphic.
See also~\cite[Theorem 3]{Gassmannleq6}, which
shows that this pair is the only nontrivial
Gassmann triple of faithful transitive $G$-sets of order $7$.
\end{example}

\subsection{L-equivalence and Gassmann equivalence}
\label{subsec:L-eq}

\hfill

Let $K_0(\Var/k)$ denote
the Grothendieck
ring of varieties with generators given by isomorphism classes
$[X]$ of schemes of finite type over $k$ and relations generated by cut and paste
relations 
\begin{equation}\label{eq:cut-and-paste}
[X] = [Z] + [X \setminus Z]
\end{equation}
for every closed $Z \subset X$.
The ring structure on $K_0(\Var/k)$ is induced by products of 
schemes.
We write $\L = [\A^1]$.

It is known that $\L$ is a zero-divisor \cite{Borisov}
and that the annihilator of $\L^k$, $k \ge 1$ encodes
deep geometric information.
Following \cite{KuznetsovShinder} we call two smooth projective
connected varieties $X$, $Y$ \emph{L-equivalent} if
for some $k \ge 0$,
\[
\L^k\cdot([X]-[Y]) = 0.
\]
We sometimes refer to the $[X] = [Y]$ case
as trivial L-equivalence.
It is currently unknown if zero-dimensional varieties 
can be nontrivially L-equivalent.
The smallest-dimensional example of nontrivial
L-equivalence is that of 
genus one curves over non-closed fields
\cite{ShinderZhang}.
See \cite{KuznetsovShinder, ShinderZhang} for some details about
conjectural relationship between L-equivalence and derived equivalence,
and the references therein for the currently known examples.
Note that the classes in the Grothendieck ring are insensitive
to nonreduced structure, hence when studying
L-equivalence we can always assume schemes to be reduced.

\begin{remark}\label{rem:alt-grothendieck}
For fields of positive characteristic, 
there exist two alternative definitions
of the Grothendieck ring of varieties,
hence alternative definitions
of L-equivalence.

(1) First of all, one can define
the Grothendieck ring  
$K_0^\bl(\Var/k)$ generated by
classes of smooth projective varieties
with blow up relations as in
\cite[Theorem 3.1 (bl)]{Bittner}.
We have an obvious homomorphism
\[
K_0^\bl(\Var/k) \to K_0(\Var/k)
\]
which is known to be an isomorphism if $k$ is a field
of characteristic zero \cite[Theorem 3.1]{Bittner}.

(2) Furthermore, in positive characteristic
one can define a modified Grothendieck ring \cite{NicaiseSebag}
by imposing an additional relation of identifying
varieties related by universal homeomorphisms
(originating from totally inseparable coverings
in positive characteristic).
It is not known if this additional relation in fact 
gives
rise to a non-isomorphic ring as all the standard invariants
which are used to distinguish elements in the Grothendieck
ring factor through
the modified ring as well \cite{NicaiseSebag}.
\end{remark}

Let $k$ be a field and 
let $G_k = \Gal(\bar{k}/k)$ where $\bar{k}$ denotes the separable closure of $k$.
For a variety $X/k$ (not necessarily smooth or projective)
we consider its $\ell$-adic
cohomology $H^i_{\et,c}(X_{\overline{k}}, \bQ_{\ell})$ (with $\ell \ne \chr(k)$)
with compact supports
as a
$G_k$-module. 
These groups are finite-dimensional $\bQ_\ell$-vector spaces
~\cite[Remark I.12.16]{FreitagKiehlEtCoh},
they vanish outside the range $0 \le i \le 2\dim(X)$ and
give rise to a group homomorphism
\begin{equation}\label{eq:mu-et}
\mu_{\et}: K_0(\Var/k) \to K_0(\Rep (G_k, \bQ_{\ell}))
\end{equation}
defined by
\[
\mu_{\et}([X]) = \sum_{i=0}^{2\dim(X)} (-1)^i [H^i_{\et,c}(X_{\overline{k}}, \bQ_{\ell})].
\]

We have
\begin{equation}\label{eq:mu-L}
\mu_{\et}(\bL) = \mu_{\et}([\bP^1]-1) = [\bQ_{\ell}(-1)]
\end{equation}
and furthermore, the projective bundle formula
for \'etale cohomology implies
that $\mu_{\et}$ is a $\Z[\L]$-module homomorphism
where $\L$ acts by multiplication by $\bQ_{\ell}(-1)$
on the Galois
representations. (The map $\mu_\et$ is even a ring homomorphism by the Künneth formula~\cite[Corollary VI.8.23]{Milne}
but we do not need this fact.)

The \'etale realization \eqref{eq:mu-et}
is useful for extracting information
from a class in the Grothendieck of varieties.
We record the following example to be used later.

\begin{example}\label{ex:mu-et-X}
Let $X$ be a geometrically rational smooth projective surface.
Then since all cohomology classes on $X$ are algebraic,
$\mu_{\et}([X]) = [\Q_\ell] + [\Q_\ell(-2)]  +
[\NS(X_{\ol{k}}) \otimes \Q_\ell(-1)]$,
where $\Q_\ell$ is considered
as a trivial one-dimensional $G_k$-representation.
In particular, if $X$, $X'$ are two
such surfaces, and $[X] = [X']$, then $\NS(X_{\ol{k}})$
and $\NS(X'_{\ol{k}})$ have the same class in $K_0(\Rep(G_k, \Q_\ell))$.
\end{example}

We explain how Gassmann equivalence relates
to L-equivalence of reduced $k$-schemes of dimension 0. Let $\EtSch_k$ be the semi-ring of $k$-schemes which are étale over $\Spec(k)$. As we assume $k$ to be perfect, $\EtSch_k$ is also the semi-ring of reduced $k$-schemes of dimension 0. Here in $\EtSch_k$, the addition (resp. multiplication) is defined by 
disjoint unions (resp. products over $\Spec(k)$). For every $Z \in \EtSch_k$, its base change $Z_{\bar{k}}$ to the separable closure $\bar{k}$ of $k$ 
is endowed with a $G_k$-action.
By Galois descent, The map $\EtSch_k \to \Gset$ sending $Z \in \EtSch_k$ to the underlying continuous $G_k$-set of $Z_{\bar{k}}$ 
is an isomorphism of semi-rings. As we assume $k$ to be perfect, this induces a ring isomorphism
\begin{equation}\label{eq:burn-galois}
 \Z[\Var^0/k] \simeq \Burn(G_k)
\end{equation}
where $\Var^0/k$ denotes the set of (irreducible) $k$-varieties of dimension 0.

We have a natural ring homomorphism
\begin{equation}\label{eq:fields-to-K0vars}
\Z[\Var^0/k] \to K_0(\Var/k)
\end{equation}
which sends $Z$ to $[Z]$. 
It follows from the blow up presentation
of the Grothendieck ring \cite[Theorem 3.1]{Bittner} that
over fields of characteristic zero \eqref{eq:fields-to-K0vars}
admits a splitting given by $X \mapsto \Spec(H^0(X,\OO_X))$ ($X$ smooth
projective), hence for
characteristic zero fields
\eqref{eq:fields-to-K0vars} is injective.

Throughout this text, the same notation $Z_{\bar{k}}$ 
(and also $Z$ itself, when it does not lead to any confusion) 
denotes the underlying $G_{k}$-set of $Z_{\bar{k}}$ for every \'etale $k$-scheme $Z$.

We work with Gassmann equivalence of such schemes
considered as sets with Galois group action.

\begin{lemma}\label{lem:L-eq-Gass}
Let $Z$ and $Z'$ be \'etale $k$-schemes. 
If $Z$ and $Z'$
are L-equivalent in the Grothendieck ring
or in any of the modifications
explained in Remark \ref{rem:alt-grothendieck}, then $Z_{\bar{k}}$ and $Z'_{\bar{k}}$
are Gassmann equivalent.
\end{lemma}

\begin{proof}
We use the \'etale realization \eqref{eq:mu-et}, which is defined
by pre-composition on $K_0^\bl(\Var/k)$,
and also factors through the modified Grothendieck
ring from Remark \ref{rem:alt-grothendieck} (2) \cite[Proposition 4.1.(3)]{NicaiseSebag}.

For any  étale $k$-scheme $Z$, we have
\begin{equation}\label{eq:mu-et-mu-G}
\mu_{\et}([Z]) = [H^0_\et(Z_{\bar{k}}, \bQ_{\ell})] = [\bQ_{\ell}[Z_{\bar{k}}]] = \mu_{G_k}([Z_{\bar{k}}]).
\end{equation}
For each $j \ge 0$ we consider the composition

	$$
		\begin{tikzcd}[cramped]
		\Burn(G_k) \simeq  \Z[\Var^0/k] \ar[r,"\times \L^j"] & K_0(\Var/k) \ar[r, "\mu_\et"]   &  K_0(\Rep(G_k,\bQ_{\ell})) \ar[r,"\otimes \bQ_{\ell}(j)"] & K_0(\Rep(G_k,\bQ_{\ell}))\\
		\end{tikzcd}
		$$
which is equal to $\mu_{G_k}$ \eqref{eq:mu}
by \eqref{eq:mu-L} and \eqref{eq:mu-et-mu-G}. Thus
if $Z$, $Z'$ are L-equivalent, then
$Z_{\bar{k}}$, $Z'_{\bar{k}}$ are Gassmann equivalent.
\end{proof}

\begin{remark}\label{rem:GCM}
Properties of classes $[X]$ of smooth projective
varieties in the Grothendieck rings often go in parallel with properties of their Chow motives \cite{Andre}.
In our situation, by \cite[Exemple
4.1.6.1]{Andre}, if $k$ has characteristic zero,
then two \'etale $k$-schemes $Z$, $Z'$ are Gassmann equivalent 
if 
and only if the Chow motives of $Z$, $Z'$ with rational
coefficients are isomorphic.
Thus in this setting 
Lemma \ref{lem:L-eq-Gass} says that
L-equivalence
implies isomorphism of Chow motives. 
The same
result is expected
for all smooth projective varieties,
and it follows from the conjectural
uniqueness of direct sum 
decompositions for Chow motives,
see e.g. 
\cite[Conjecture 2.5, Conjecture 2.6]{Goettsche-hilbert}.
\end{remark}

\begin{example}\label{ex:R-Fq}
Let $k$ be a field
such that all continuous $G_k$-actions
on finite sets factor through a
finite cyclic quotient of $G_k$, e.g.
$k = \R$ or $k = \F_q$.
By Proposition \ref{prop:Gassmann} cyclic groups
do not allow nontrivial Gassmann equivalence, 
hence by Lemma \ref{lem:L-eq-Gass}
L-equivalence of \'etale $k$-schemes implies
their isomorphism.
\end{example}

\begin{corollary}\label{cor:L-equiv}
Let $Z$ and $Z'$ be \'etale $k$-schemes. 
If $Z$ and $Z'$
are L-equivalent 
in the Grothendieck ring
or any of the modifications explained in Remark \ref{rem:alt-grothendieck}, then $Z \simeq Z'$
as soon as one of the following
conditions is satisfied:
\begin{itemize}
\item $Z$ is $k$-irreducible and 
Galois over $\Spec(k)$
\item $Z$ is $k$-irreducible with $\deg_k Z \le 6$
\item $Z$ satisfies $\deg_k Z \le 5$.
\end{itemize}
\end{corollary}
\begin{proof}
By Lemma \ref{lem:L-eq-Gass}, $Z_{\bar{k}}$ and $Z'_{\bar{k}}$
are Gassmann equivalent
and by 
Proposition \ref{prop:Gassmann} they are isomorphic.
\end{proof}

\begin{example}\label{ex:ex-cubics}
Translating Example \ref{ex:gassmann}(1) into the language of \'etale schemes, we obtain
the following. Let $k = \Q$,
and define degree $6$ schemes
\begin{align*}
Z &= \Spec(\Q(\sqrt{\alpha})) \sqcup \Spec(\Q(\sqrt{\beta})) \sqcup \Spec(\Q(\sqrt{\alpha\beta})) \\
Z' &= \Spec(\Q(\sqrt{\alpha}, \sqrt{\beta})) \sqcup \Spec(\Q) 
\sqcup \Spec(\Q),
\end{align*}
where we choose any $\alpha, \beta \in \Q^\star$ which are nontrivial and distinct in $\Q^\star / (\Q^\star)^2$.
Then $Z$, $Z'$ are Gassmann equivalent, and have isomorphic Chow motives
(see Remark \ref{rem:GCM})
but we do not know how to check
if they are L-equivalent or not.

Embedding $Z$, $Z'$ into $\P^2$ so that both
images are in general position (for instance, if $\ga+\gb\neq1$, we can send $Z$ onto $[\pm\sqrt{\alpha}:1:0],[0:\pm\sqrt{\beta}:1],[1:0:\pm\sqrt{\alpha\beta}]$ and $Z'$ onto $[\pm\sqrt{\alpha}:\pm\sqrt{\beta}:1],[1:0:1],[0:1:1]$), we obtain two del Pezzo surfaces of degree $3$:
$X = \Bl_{Z}(\P^2)$, $X' = \Bl_{Z'}(\P^2)$.
These two surfaces are not isomorphic, as the first one has three
$k$-rational lines, and the second one has five.

However, we have an isomorphism
\[
\NS(X_{\bar{\Q}}) \otimes \Q \simeq \NS(X'_{\bar{\Q}}) \otimes \Q
\]
of permutation Galois representations, which shows that
the associated Galois set is not uniquely defined.
One can make a similar example integrally, using
integral
Gassmann triples
as in \cite{Scott}, \cite{PrasadGassmann}. 
See Example \ref{ex:cubics-M} where
we explain that using our techniques
we \emph{can} recover $Z$ (resp. $Z'$)
from $X$ (resp. $X'$), providing another proof
that they are not isomorphic.

Finally note that we do not know if $[X] = [X']$ as by the blow up relation
in the Grothendieck ring of varieties,
$[X] = [\P^2] + \L[Z]$ and $[X'] = [\P^2] + \L[Z']$ so 
$[X] = [X']$ 
would imply the
L-equivalence of $Z$ and $Z'$, which is unknown.
\end{example}

\section{Factorization centers}\label{sec:main}

In this section we introduce the invariant $c(\phi)$ 
keeping track of the factorization centers  in any decomposition
of a birational isomorphism of $\phi$ between surfaces
into a sequence of blow ups
and blow downs,
formulate the main theorem and  interpret it
in terms of the truncated
Grothendieck ring of varieties.

\subsection{Formulation of the main result}

Fix a perfect field $k$.
Let $\phi: X \dashrightarrow Y$ be a birational
isomorphism of smooth projective $k$-surfaces.
By the strong factorization theorem~\cite[Corollary of Lemma III.4.4]{Manin}
(or~\cite[Lemma 54.17.2]{stacks-project})
we have a decomposition
\begin{equation}\label{eq:phi}
\xymatrix{
& \wt{X} \ar[dl]_{\alpha} \ar[dr]^{\beta} &\\
X \ar@{-->}[rr]^\phi & &  Y  
}
\end{equation}
with both $\alpha$ and $\beta$ being compositions of blow ups
with smooth centers $Z_1, \dots, Z_r$ and $Z_1', \dots, Z_{s}'$ respectively, which are 
zero-dimensional smooth schemes. 
The \emph{factorization center} of $\phi$ is defined as
\begin{equation}\label{eq:def-c}
c(\phi) = \sum_{i=1}^{r} [Z_i] - \sum_{i=1}^{s} [Z_i'] \in \Z[\Var^0/k]
\end{equation}
(see \S \ref{subsec:L-eq} for the definition of $\Z[\Var^0/k]$).
We remark that $\L\cdot c(\phi)$,
regarded as an element of $K_0(\Var/k)$,
measures the difference between classes $[Y]$ and $[X]$ in the Grothendieck ring of varieties, see Lemma \ref{lem:c-Groth}.
Given that $\L$ is a zero-divisor in the Grothendieck ring, 
$c(\phi)$  potentially contains
more information than $\L \cdot c(\phi)$.

\medskip

We explain the well-definedness of $c$
and its basic properties. 
To do that it is most convenient to consider the
groupoid $\Bir_2/k$
of birational types
of surfaces, whose objects are smooth
projective surfaces and morphisms are birational
isomorphisms.

Recall that if $\CC$ is a groupoid,
and $G$ a group, a homomorphism 
from $\CC$ to $G$ is a functor
from $\CC$ to $G$, where $G$ is considered
as a groupoid
with one object.

\begin{lemma}\label{lem:c-add}
$c(\phi)$ does not depend on the
choice of factorization of $\phi$ and 
defines a homomorphism
$c: \Bir_2/k \to \Z[\Var^0/k]$.
Explicitly, for any two birational isomorphisms of surfaces $\phi: X \dashrightarrow X'$,
$\psi: X' \dashrightarrow X''$ we have
\[
c(\psi \circ \phi) = c(\psi) + c(\phi).
\]
In particular, for any surface $X$ we have a homomorphism
\[
c: \Bir(X) \to \Z[\Var^0/k]. 
\]
\end{lemma}
\begin{proof}
Consider the diagram \eqref{eq:phi}.
Let $E_1, \dots, E_m \subset X$ (resp. $E_1', \dots, E_n' \subset Y$)
be the irreducible components of the exceptional divisor of $\phi$ (resp. $\phi^{-1}$).
Let $D_1, \dots, D_t \subset \wt{X}$ 
be the irreducible divisors
which are contracted by both $\alpha$ and $\beta$. 
This way the centers $Z_i, \; i = 1, \dots, r$ 
(resp. $Z_i', \; i = 1, \dots, r'$)
are in one-to-one correspondence
with the 
collection of
divisors $\{E_i\}_{i = 1, \dots, m} \cup \{D_j\}_{j = 1, \dots, t}$
(resp. $\{E'_i\}_{i = 1, \dots, n}\cup \{D_j\}_{j = 1, \dots, t}$).

Note that 
each prime divisor $D$ contracted by
 $\alpha$ or $\beta$ is birational to $\P^1 \times Z$,
where $Z$ is the center of the corresponding blow up,
which can be recovered from $D$ as $Z \simeq \Spec(H^0(D, \OO_D))$.
Cancelling out the centers corresponding to
$D_1, \dots, D_t$ we obtain
\[\begin{aligned}
c(\phi) = \sum_{i=1}^{r} [Z_i] - \sum_{i=1}^{s} [Z_i'] = 
\sum_{i=1}^{m} [\Spec(H^0(E_i, \OO))] - \sum_{i=1}^{n} [\Spec(H^0(E_i', \OO))] \\
\end{aligned}\]
thus the expression \eqref{eq:def-c}
only depends on the exceptional divisors of $\phi$
and $\phi^{-1}$, hence $c(\phi)$ is independent of
the choice of strong factorization \eqref{eq:phi}.

To show that $c(\psi \circ \phi) = c(\psi) + c(\phi)$
consider the diagram
			$$
		\begin{tikzcd}[cramped]
		& &  X_3 \ar[dl, "\mu", swap] \ar[dr, "\mu'"]  & & \\
		&  X_1 \ar[dl, "\tau", swap] \ar[dr, "\tau'"] &  &  X_2 \ar[dl, "\nu", swap]   \ar[dr, "\nu'"] & \\
		X \ar[rr, dashed, "\phi"] &  & X' \ar[rr, dashed, "\psi"] & & X'' 
		\end{tikzcd}
		$$
where $\tau$, $\tau'$ (resp. $\nu$, $\nu'$) provide a strong
factorization for $\phi$ (resp. $\psi$),
and $\mu$, $\mu'$ provide a strong factorization
for $\nu^{-1} \tau'$.
It is clear that $c$ is additive on regular
birational isomorphisms, hence
		\begin{equation*}
		\begin{split}
		c(\psi \circ \phi) & = c(\nu' \circ \mu') - c(\tau \circ \mu) \\
		& = c(\nu' \circ \mu') - 
		c(\nu \circ \mu') + c (\tau' \circ \mu) - c(\tau \circ \mu) \\
		 & = c(\nu') - c(\nu) + c(\tau') - c(\tau) = c(\psi) + c(\phi).
		\end{split}
		\end{equation*}
\end{proof}

\begin{corollary}\label{cor:c-axioms}
There is a unique assignment $\phi \mapsto c(\phi) \in \Z[\Var^0/k]$ 
defined for all birational isomorphisms $\phi$ 
between smooth projective $k$-surfaces, satisfying
the following axioms
\begin{enumerate}[(i)]
    \item For any (biregular) 
    isomorphism $\phi$, $c(\phi) = 0$
    \item If $\wt{X}$ is a blow up of $X$ with smooth connected center 
    $Z$ and $\phi: \wt{X} \to X$ the contraction map, then
    \[
    c(\phi) = -[Z], \; c(\phi^{-1}) = [Z]
    \]
    \item For composable birational isomorphisms
    \[
    c(\psi \circ \phi) = c(\psi) + c(\phi).
    \]
\end{enumerate}
\end{corollary}
\begin{proof}
Uniqueness follows from the strong factorization \eqref{eq:phi},
and existence is the content of Lemma \ref{lem:c-add}.
\end{proof}

\begin{remark}
The invariant $c(\phi)$ is generalized to higher-dimensional varieties
in \cite{BirMot},
to the equivariant setting in \cite{KreschTschinkel}
and to varieties with a logarithmic form in \cite{CLKT}.
In all these contexts properties of $c(\phi)$ are analogous to those stated in
Corollary \ref{cor:c-axioms}.
\end{remark}

\begin{proposition}\label{prop:c-Galois}
\begin{enumerate}[(i)]
    \item For any field extension $L/k$ we have a commutative diagram
\[\xymatrix{
\Bir(X,Y) \ar[d] \ar[r]^c & \Z[\Var^0/k] \ar[d] \ar[r]^\simeq_{\eqref{eq:burn-galois}} & \Burn(G_k) \ar[d] \\
\Bir(X_L,Y_L) \ar[r]^c & \Z[\Var^0/L] 
\ar[r]^\simeq_{\eqref{eq:burn-galois}} & \Burn(G_L) \\
}\]
where the left vertical arrow is an extension of scalars
for birational map, the middle vertical arrow
is defined by extension of scalars, that is it
maps an \'etale $k$-scheme $Z$ to the sum of the connected
components of $Z \times_k L$,
and the right vertical arrow is 
the restriction of the group action
defined through the map  
$\Gal(\bar{L}/L) \to \Gal(\bar{k}/k)$ 
induced by any choice of embedding 
$\bar{k} \hto \bar{L}$ compatible with $k \hto L$.

In particular, applying the diagram to an automorphism
$\sigma \in \Aut(L/k)$
the left square gives
$c(\sigma(\phi)) = \sigma(c(\phi))$, that is
$c$ commutes with 
the group action by the automorphisms of the field.

\item For any finite field extension $L/k$ and
surfaces $X$, $Y$ over $L$, 
let $X|_k$, $Y|_k$ denote the underlying $k$-surfaces.\footnote{Here by $X|_k$
we mean the $k$-surface given by the composition
$X \to \Spec(L) \to \Spec(k)$ (this is not the Weil restriction
of scalars). If $L/k$ is nontrivial, the 
surface $X|_k$ is not geometrically connected:
$X|_k \times_k L$ is isomorphic to
a disjoint union of $[L:k]$ copies of $X$.}
We have the following commutative diagram:
\[\xymatrix{
\Bir(X,Y) \ar[d] \ar[r]^c & \Z[\Var^0/L] \ar[d] \ar[r]^\simeq_{\eqref{eq:burn-galois}} & \Burn(G_L) \ar[d] \\
\Bir(X|_k,Y|_k) \ar[r]^c & \Z[\Var^0/k]
\ar[r]^\simeq_{\eqref{eq:burn-galois}} & \Burn(G_k) \\
}\]
where the middle vertical map is restriction of scalars.
\end{enumerate}
\end{proposition}
\begin{proof}
By Lemma \ref{lem:c-add},
in both (i), (ii) it is sufficient to check
a single blow up where the statements are clear.
\end{proof}

The main result of the paper is the following.

\begin{theorem}\label{thm:main}
For any two smooth projective $k$-surfaces $X$, $Y$
and any two birational isomorphisms $\phi, \psi: X \dashrightarrow Y$,
we have $c(\phi) = c(\psi)$.
In particular, $c: \Bir(X) \to \Z[\Var^0/k]$ is a zero map.
\end{theorem}

We prove Theorem \ref{thm:main} at the end of \S \ref{sec:geom-rat-5}.
The result is straightforward when $k$ is algebraically closed,
because in this case the invariant
takes values in $\Z\cdot[\Spec(k)] \simeq \Z$,
and measures the
difference of the ranks
of two N\'eron-Severi groups
(cf Proposition \ref{prop:c-nongeomrat}).
However, Theorem \ref{thm:main} 
becomes a nontrivial statement when $k$ is an arbitrary perfect field,
and its proof depends on the two-dimensional
Minimal Model Program.
This result is also specific for surfaces and
fails to be true in higher dimension, even over
algebraically closed fields 
\cite{BirMot}.

\begin{example}\label{ex:R1}
Let $k = \R$. In this case Theorem \ref{thm:main} says
that birational automorphisms of surfaces
blow up the same number of rational points as 
the number of
rational divisors they blow down, 
and that they blow up the same number
of pairs of complex conjugate points as the
number of
pairs of complex conjugate divisors they blow down.
This can be proved directly by considering
the Galois action on the N\'eron-Severi group
(see also Example \ref{ex:R-Fq}).
\end{example}

\begin{remark}
A different proof of Theorem \ref{thm:main}
in the case when the surfaces $X$, $Y$
are rational can be deduced from a more recent result by Lamy and Schneider~\cite{LamySchneider} 
who analyze relations between
Sarkisov links to prove that 
$\Cr_2(k)$ is generated by involutions. This implies that every homomorphism from $\Cr_2(k)$ to a free abelian group is trivial
(for $k = \R$ one can also deduce this from 
\cite[Theorem 1.1]{Zim}).
On the other hand this is not true for other types
of del Pezzo surfaces:
Blanc, Schneider and Yasinsky
have constructed
nontrivial homomorphisms
$\Bir(X) \to \Z$
for Severi-Brauer surfaces $X$
\cite{BSY}.
\end{remark}

\begin{example}\label{ex:dP5-links}
Consider
the following composition $\phi$ 
of type $II$ links (see Definition \ref{def:links} for links)
between del Pezzo surfaces $dP_{d}$ (where $d = K^2$ stands for the degree) 
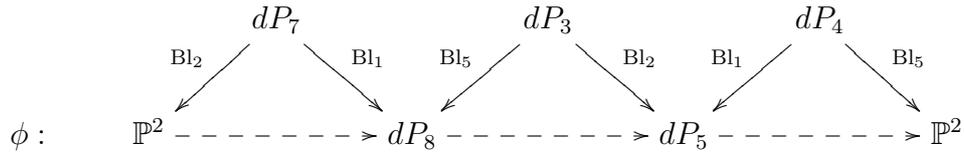
\[
\xymatrix{
 &  & dP_7\ar[dr]^{\Bl_{1}} \ar[dl]_{\Bl_{2}} &   & dP_3 \ar[dr]^{\Bl_{2}} \ar[dl]_{\Bl_{5}}&   & dP_4 \ar[dr]^{\Bl_{5}} \ar[dl]_{\Bl_{1}}&  \\
\phi: & \P^2 \ar@{-->}[rr] &   & dP_8 \ar@{-->}[rr]  &  & dP_5 \ar@{-->}[rr] &   & \P^2 \\
}
\]
where each map $\Bl_i$ is the blow up (resp. blow down) along a Galois orbit $Z_i$ (resp. $Z'_i$) of degree $i$.
See \cite{IskovskikhFact} for the general
results on links, or see the explicit constructions
explained below for the existence of the three links above.
We have
\[
c(\phi) = [Z_2]-[k] + [Z_5] - [Z_2'] + [k]-[Z_5'] = [Z_2]-[Z_2'] + [Z_5]-[Z_5'],
\]
so that
$c(\phi) = 0$ 
is equivalent to
$Z_2 \simeq Z_2' \text{ and } Z_5 \simeq Z_5'$.

One way to understand this is to read 
the diagram as a composition of two different well-known
rationality
constructions for del Pezzo surace $dP_5$: one blows up
a point $P$ and contracts a Galois orbit of five $(-1)$-curves
which are proper preimages of conics through $P$  
(connecting $dP_5$ with $\P^2$ moving right)
or one blows up a Galois orbit
of two points $Q_1$, $Q_2$
and contracts five $(-1)$-curves obtained
as proper preimages of 
cubics 
passing through $Q_1$, $Q_2$
onto a rational quadric $dP_8$ and then transforms
it to $\P^2$ (connecting $dP_5$ with $\P^2$ moving left).
This way $dP_5$ has two potentially different rationality
centers that can be blown up by $\P^2 \dashrightarrow dP_5$,
namely $Z_5$ and $Z_5'$ and the result is that these
centers are in fact isomorphic (see Corollary \ref{cor:rat-centers} and Example \ref{ex:dP5-center}).
\end{example}

\subsection{Interpretation in terms of the Grothendieck ring of varieties}\label{subsec:Groth}
\hfill

From the perspective of the Grothendieck
ring of varieties, we have
the following interpretation of $c(\phi)$.

\begin{lemma}\label{lem:c-Groth}
For any birational isomorphism $\phi: X \dashrightarrow Y$ between
smooth projective surfaces
we have the following identity in $K_0(\Var/k)$
\begin{equation}\label{eq:groth}
[Y] = [X] + \L \cdot \ol{c}(\phi), 
\end{equation}
where $\ol{c}$ is the composition
of $c$
with the natural homomorphism \eqref{eq:fields-to-K0vars} $\Z[\Var^0/k] \to K_0(\Var/k)$.
\end{lemma}
\begin{proof}
Using Lemma \ref{lem:c-add} we see that
\eqref{eq:groth} is preserved
under compositions of birational isomorphisms.
Since birational isomorphisms
are decomposed
into blow ups and blow downs along $k$-étale subschemes,
it suffices to check Lemma~\ref{lem:c-Groth} for a single blow up,
where the result is clear
by Corollary \ref{cor:c-axioms}(2)
and the blow up formula in the Grothendieck ring.
\end{proof}

\begin{remark}\label{rem:L-eq}
We see that
$\L \cdot \ol{c}(\phi) \in K_0(\Var/k)$
only depends on $[X]$ and $[Y]$, 
however 
it is known that $\L$ is a zero-divisor \cite{Borisov}
so a priori
we cannot divide $\L \cdot \ol{c}(\phi)$ by $\L$ and deduce
that $\ol{c}(\phi)$ 
(or $c(\phi)$)
only depends
on $X$ and $Y$ but not on $\phi$.
\end{remark}

We can informally reformulate Theorem \ref{thm:main}
by the statement
that $0$-dimensional $k$-varieties
cannot be \emph{L-equivalent via surfaces}.
Let us explain this. 
For each $n \ge 0$
consider 
$K_0(\Var^{\le n}/k)$,
the abelian group generated
by isomorphism classes of varieties of dimension $\le n$,
modulo cut and paste relations \eqref{eq:cut-and-paste}.
For each
$n \le m$ there is a (generally
non-injective) group homomorphism
\[
K_0(\Var^{\le n}/k) \to K_0(\Var^{\le m}/k)
\]
and $K_0(\Var/k)$, as an abelian group,
is the colimit of this system.
For every $n$ there is a surjective homomorphism
\[
K_0(\Var^{\le n}/k) \to \Z[\Bir_{n}/k],
\]
where $\Bir_n/k$ is the set of $k$-birational classes
of dimension $n$.
The kernel of this homomorphism is spanned
by varieties of dimension $\le (n-1)$, but 
it may not
be isomorphic to $K_0(\Var^{\le n-1}/k)$;
see \cite{Zakharevich} for interpretation
of this in terms of an algebraic
$\Kt$-theory spectral sequence.

For $n=0$, $K_0(\Var^{\le 0}/k)$
is canonically isomorphic to $\Z[\Var^0/k]$,
and $K_0(\Var^{\le 1}/k)$ fits into a split exact
sequence
\begin{equation}\label{eq:K0Groth-1}
0 \to \Z[\Var^0/k] \to K_0(\Var^{\le 1}/k) \to
\Z[\Bir_1/k] \to 0.
\end{equation}

Our main results can be reformulated as results
about $K_0(\Var^{\le 2}/k)$:

\begin{corollary}[of Theorem \ref{thm:main}]\label{cor:main}
(i) We have a short exact sequence
\begin{equation}\label{eq:K0-2seq}
0 \to \Z[\Var^0/k] \oplus \Z[\Bir_1/k] \to
K_0(\Var^{\le 2}/k) \to \Z[\Bir_2/k] \to 0,
\end{equation}
where the last map sends a combination of varieties to the birational
classes of its $2$-dimensional components,
and the first map sends a $k$-variety $Z$ of dimension $0$ to $[Z]$
and a birational class of a curve to the class of its unique
smooth projective model.

(ii) We have a (noncanonical) splitting
\[
K_0(\Var^{\le 2}/k) \simeq 
\Z[\Var^0/k] \oplus \Z[\Bir_1/k] 
\oplus \Z[\Bir_2/k],
\]
and in particular
$K_0(\Var^{\le 2}/k)$ is a (torsion-)free abelian group.

(iii) If $Z$ and $Z'$ are \'etale $k$-schemes,
and \[\L^i \cdot ([Z] - [Z']) = 0\] in 
$K_0(\Var^{\le 2}/k)$  {with $i \in \{0,1,2\}$} 
(where $\L^i \cdot ([Z] - [Z'])$ is represented by $[\A^i \times_k Z] - [\A^i \times_k Z']$), then $Z$ and $Z'$ are isomorphic.
\end{corollary}
\begin{proof}
(i) 
Because resolution of singularities and weak factorization
(even strong factorization) are
known for surfaces over arbitrary 
perfect fields, the proof of \cite{Bittner} goes through to show that
we have an isomorphism
\[
K_0^\bl(\Var^{\le 2}/k) \to K_0(\Var^{\le 2}/k),
\]
where the first group is defined by 
smooth projective varieties
of dimension $\le 2$ and 
Bittner's blow up relations.
These relations can be equivalently
presented as 
\begin{equation}\label{eq:blowup-rel-c}
[Y] = [X] + \L \cdot \bar{c}(\phi)
\end{equation}
for all birational 
isomorphisms $\phi: X \dashrightarrow Y$
between smooth projective surfaces. Indeed,
\eqref{eq:blowup-rel-c} include the blow up
relations if $\phi: Y \to X$ is a smooth blow up
as a particular case
and conversely,
\eqref{eq:blowup-rel-c} follow 
as soon as we impose
the blow up relations as
in the proof of Lemma
\ref{lem:c-Groth}.
Since curves admit unique smooth projective models, we have
the canonical splitting of \eqref{eq:K0Groth-1} giving
\[
K_0(\Var^{\le 1}/k) = K_0^\bl(\Var^{\le 1}/k) \simeq \Z[\Var^0/k]
\oplus \Z[\Bir_1/k].
\]
Furthermore we have an obvious
short exact sequence
\[
K_0^\bl(\Var^{\le 1}/k) 
\to K_0^\bl(\Var^{\le 2}/k)
\to \Z[\Bir_2/k] \to 0
\]
and to prove \eqref{eq:K0-2seq} it suffices to
show that the first map in the sequence is
split-injective.

The splitting is based on 
Theorem \ref{thm:main} and is not
canonical. First of all,
we choose a smooth projective representative
for each birational class of surfaces.
If $L$ is a two-dimensional function field,
we write $M_L$ for the chosen model.
We define the splitting
\[
\epsilon: K_0^\bl(\Var^{\le 2}/k) \to K_0^\bl(\Var^{\le 1}/k)
\]
by identity on classes of smooth projective curves
and zero-dimensional schemes and if $X$ is 
a smooth projective surface we define
\[
\epsilon([X]) = \L \cdot \ol{c}(M_{k(X)} 
\overset{\psi}{\dashrightarrow}
X)
\]
for any choice of a birational isomorphism $\psi$
between $X$ and its model $M_{k(X)}$.
The fact that this is independent of $\psi$
is the content of Theorem \ref{thm:main},
and the fact that $\epsilon$ is well-defined,
that is preserves the relations
\eqref{eq:blowup-rel-c},
follows immediately from the property
that $c$ is additive on compositions
(Lemma \ref{lem:c-add})
applied to a composition of birational
isomorphisms
$M \overset{\psi}{\dashrightarrow} X 
\overset{\phi}{\dashrightarrow} Y$
(where $M := M_{k(X)} = M_{k(Y)}$).

(ii) follows from \eqref{eq:K0-2seq};
explicit splittings are constructed in
the proof of (i).

(iii) The case $i = 2$ is clear as if $\A^i \times_k Z$
and $\A^i \times_k Z'$ are $k$-birational, then $Z$ and $Z'$
are isomorphic.
For $i = 1$, the element
$\L \cdot ([Z] - [Z'])$
is the image of $(-[Z]+[Z'], [\P^1 \times_k Z] - [\P^1 \times_k Z'])$ under the first map
in \eqref{eq:K0-2seq}, and since this map is injective, $[Z] = [Z']$,
that is $Z$ and $Z'$ are isomorphic.
Finally the case $i = 0$ follows again from \eqref{eq:K0-2seq}.
\end{proof}

\section{Birational geometry of surfaces}\label{sec:birsurf}

In this section we recall the Minimal Model Program
and Sarkisov link decomposition for surfaces,
which is used in our proof of Theorem \ref{thm:main}.
In Proposition \ref{prop:c-nongeomrat} we prove
Theorem \ref{thm:main} for birational types
with particularly simple links.
In \S \ref{subsec:curves-dP} we investigate
linear systems of rational curves on del Pezzo surfaces,
which will be needed to finish
the proof of Theorem \ref{thm:main}
in \S \ref{sec:geom-rat-5}.

\subsection{Birational classification of surfaces and links}

Let $X/k$ be a geometrically irreducible surface. 
Recall that $X$
is called minimal
if it does not have a Galois orbit
of disjoint $(-1)$-curves, or equivalently,
every regular birational map $X \to X'$
from $X$ to a smooth surface is an isomorphism.
We say that $X$ is rational if it is birational to $\P^2$
over $k$ and $X$ is geometrically rational
if it is birational to $\P^2$ over $\ol{k}$.
A del Pezzo surface is a smooth projective
geometrically irreducible surface with ample anticanonical class.

A conic bundle is a fibered surface $\pi: X \to C$, with $C$
a smooth projective curve, such that the generic fiber of $\pi$ 
is a smooth rational curve
and $X$ has Picard rank two. 
For a conic bundle  $\pi: X \to C$, $\pi_{*}\go_{X/C}^\vee$ is locally free of rank $3$  and we have an embedding $X \hto \P(\pi_{*}\go_{X/C}^\vee)$ over $C$ realizing each fiber of $\pi: X \to C$ as a plane conic (see~\cite[Corollary 3.7]{HassettNC}). 
The number $m$
of geometric singular fibers of a conic bundle
equals
\begin{equation}\label{eq:KX2-conicbundle}
m = 8 - K_X^2
\end{equation}
see e.g. \cite[Theorem 3]{IskovskikhMinModel},
in particular $K_X^2 \le 8$.

We write $\NS(X)$
for the N\'eron-Severi group
of divisors modulo algebraic equivalence; by
the theorem of N\'eron-Severi it is a
finitely generated abelian group.
Each contraction
of a Galois orbit of $(-1)$-curves
decreases the rank of the N\'eron-Severi group,
hence a sequence of contractions 
always terminates
to produce a minimal surface birationally
equivalent to the given one.

We have the following classification result going back
to Enriques, Manin and Iskovskikh, see \cite[Theorem III.2.2]
{Kollarrat}.

\begin{theorem}[Minimal Model Program in dimension two]
\label{thm:mmp}
Any geometrically irreducible minimal
surface $X/k$ is isomorphic to exactly one of the following:
\begin{itemize}
    \item Non-geometrically rational case
\begin{enumerate}
    \item[(1)] Surface with nef $K_X$  
    \item[(2)] Conic bundle $\pi : X \to C$ with
     $g(C) > 0$
\end{enumerate}

    \item Geometrically rational case
\begin{enumerate}
    \item[(3)] Conic bundle $X \to C$ with $g(C) = 0$
    \item[(4)] Minimal del Pezzo surface of Picard rank one
\end{enumerate}
\end{itemize}

In cases (2) and (3), $\NS(X)$ is of rank two.

\end{theorem}

Note that if $k$ is algebraically closed,
any conic bundle in (2), (3) has
no singular fibers (as components of singular
fibers are $(-1)$-curves),
thus (2) consists of ruled surfaces
and (3) consists of Hirzebruch
surfaces $\F_n$.

\begin{definition}[Sarkisov links, see 2.2 in \cite{IskovskikhFact}]\label{def:links}
 Let $X \to B$ and $X' \to B'$ 
 be Mori fiber spaces with $\dim X = \dim X' = 2$. 
Here each Mori fiber space is either a del Pezzo surface
 of Picard rank one or a conic bundle of Picard rank two over
 a smooth curve.
 A Sarkisov link between $X$ and $X'$ 
 is a birational map $\nu : X \dto X'$ 
 satisfying one of the following descriptions. 
 \begin{itemize}
    \item \textbf{(Type I)}
    $\nu^{-1}$ is the blow up of $X$ at a smooth
    closed
    point with $B = \Spec(k)$ and $X' \to B'$ a conic bundle.
    \item \textbf{(Type II)} We have a commutative diagram
      \begin{equation}\label{cd-TypeII}
          \begin{tikzcd}[cramped]
		 X \ar[d] & Y \ar[l, swap,"\ga"] \ar[r,"\gb"] & X' \ar[d]\\
		B \ar[rr,"\sim"] &  & B'
		\end{tikzcd}
      \end{equation}
      and $\nu = \gb \circ \ga^{-1}$ where both $\ga$ and $\gb$ are blow ups at a smooth closed point.
      In this case,
      \begin{itemize}
          \item \textbf{(Type IIC)} either both
          $X \to B$ and $X' \to B'$
          are conic bundles;
          \item \textbf{(Type IID)} or $B = B' = \Spec(k)$. 
      \end{itemize}
    \item \textbf{(Type III)} This is the inverse of link (I).
    \item \textbf{(Type IV)} $X \rightarrow B$, $X' \to B'$ are conic bundles and $\nu$ is an isomorphism not respecting the conic bundle structures.
 \end{itemize}
\end{definition}

We rely on decomposing birational isomorphisms
of surfaces into Sarkisov links \cite{IskovskikhFact}. Some
of the links are 
easier
to deal with.

\begin{lem}\label{lem:cIIC=0}
If $\nu: X \dashrightarrow X'$ is a Sarkisov link of type IIC, then $K_{X}^2 = K_{X'}^2$,
$\rk \; \NS(X)=\rk \; \NS(X')$ and  $c(\nu) = 0$.
\end{lem}
\begin{proof}
Let $\phi : X \to B$ be the conic bundle as in~\eqref{cd-TypeII}. By~\cite[Theorem 2.6]{IskovskikhFact}, $\nu$ is an elementary transformation of $X \to B$. More precisely, $\ga$ in~\eqref{cd-TypeII} is the blow up at a smooth closed
point $p \in X$ lying in a smooth fiber $C$ of $\phi$ and $\gb$ in~\eqref{cd-TypeII} is the contraction of the proper transformation $\ti{C}$ of $C$ under $\ga$. Since $\ti{C} \simeq C \simeq \P^1 \times_{\Spec(k)} \{p\}$, the blow up center of $\gb$ is isomorphic to $p$. Hence $\rk \; \NS(X)=\rk \; \NS(X')$ and $c(\nu) = 0$.
The equality $K_{X}^2 = K_{X'}^2$ follows
from the fact that $X \to B$ and $X' \to B'$ have the same number
of geometric singular fibers using \eqref{eq:KX2-conicbundle}.
\end{proof}

Recall that the degree of a geometrically rational surface $X$ is defined to be $K_X^2$.
The next result is a step
towards Theorem \ref{thm:main}.

\begin{proposition}\label{prop:c-nongeomrat}
Let $X$, $X'$ be a pair of birational
minimal 
geometrically irreducible
surfaces.
Assume that $X$
is either (1) geometrically irrational
or (2) geometrically rational 
and of degree $\le 4$,
then the same holds for $X'$
and for any birational isomorphism
$\phi: X \dashrightarrow X'$ we have 
\begin{equation}\label{eq:c-rkNS}
c(\phi) = (\rk \; \NS(X') - \rk \; \NS(X)) \cdot [\Spec(k)].
\end{equation}
In particular $c(\phi) = 0$ for birational automorphisms
of such surfaces.
\end{proposition}
\begin{proof}
(1) It is clear
that $X'$ is geometrically irrational and moreover, $X$ and $X'$ have the same type
in Theorem \ref{thm:mmp}.
If $X$, $X'$ have nef canonical class, then
every birational isomorphism is a biregular
isomorphism by \cite[Corollary 1 in II.7.3]{IskovskikhShafarevich-Surfaces}
so that $X \simeq X'$ and $c(\phi) = 0$ as required.

If $X$, $X'$ are conic bundles over
a curve of positive genus, then by \cite[Corollary 3.2]{Sch19},
birational maps between them can be decomposed
into Sarkisov links of type IIC, in which case
$\rk \; \NS(X) = \rk \; \NS(X')$ and $c(\phi) = 0$ by Lemma~\ref{lem:cIIC=0}.

(2) 
It follows from the classification of links in \cite[Theorem 2.6]{IskovskikhFact} that
elementary links will only connect
$X$ to 
minimal surfaces of degree $\le 4$
or conic bundles of degree $3$ and Picard rank $2$,
and that
birational isomorphisms between such surfaces
will be decomposed into
\begin{itemize}
    \item Biregular isomorphisms
    \item Type IIC links 
    \item Bertini and Geiser involutions
    \item Blow ups of a rational point (between degree $4$
    and degree $3$ surfaces).
\end{itemize}

For each of these types the claim of Proposition~\ref{prop:c-nongeomrat} is true, namely in the first three
cases $\rk \; \NS(X) = \rk \; \NS(X')$ and $c(\phi) = 0$
(using Lemma \ref{lem:cIIC=0} for the second case),
while in the last case the ranks differ by
one and the result is true by definition of $c$.
Equality~\eqref{eq:c-rkNS} follows by additivity of $c$
under compositions (Lemma~\ref{lem:c-add}).
\end{proof}

\begin{remark}
From the proof of Proposition~\ref{prop:c-nongeomrat}, 
we note that in~\eqref{eq:c-rkNS}, we have actually
$c(\phi) = 0$ or $c(\phi) = \pm [\Spec(k)]$. 
Moreover, $c(\phi) = \pm [\Spec(k)]$ only happens 
when $\phi$ is a map between a del Pezzo surface of degree $4$
and a conic bundle of degree $3$.
\end{remark}

\subsection{Rational curves on del Pezzo surfaces}
\label{subsec:curves-dP}

Linear systems 
of rational curves on del Pezzo surfaces
are closely related to factorization centers:
for instance, to create a rational two-dimensional
quadric $X$, one needs to blow up $\P^2$ in $Z_2$,
where $Z_2$ is degree two étale $k$-scheme, and to contract
a line through the center. This way the original
scheme $Z_2$ can be recovered as the scheme
parametrizing rulings on $X$ (cf Definition \ref{def:MX}).

\begin{definition}\label{def-deg}
Let $j \ge 1$.
We call a complete
linear system $|L|$ of curves on a del Pezzo surface $X$
\emph{a linear system of degree $j$ rational curves}
if a general member $C \in |L|$ is a smooth rational
curve and $(-K_X) \cdot L = j$.
\end{definition}

For each $j \ge 1$ we consider $\HH^j(X)$,
the Hilbert scheme of curves on $X$ with general members
of each component
being 
smooth rational curves of degree $j$. 
By an easy computation
(see Lemma \ref{lem:curves}(i))
$H^1(X, L) = 0$, hence $\HH^j(X)$ is smooth \cite[Theorem 2.8]{Kollarrat}.
In fact, over the algebraic
closure $\ol{k}$, each $\HH^j(X_{\ol{k}})$ is a disjoint union of projective
spaces parametrizing effective divisors in the corresponding
linear systems.
We refer to points of $\HH^1(X)$ as lines on $X$,
points of $\HH^2(X)$ as conics and so on.
When $X$ is a twisted form of $\P^1 \times \P^1$
(a minimal del Pezzo surface of degree $8$), then 
all curves have even degree,
and families of conics on $X$ are also
called rulings.

We note that by adjunction formula, a linear system
$L$ of rational curves 
of degree $j$ satisfies
\begin{equation}\label{eq:num-rat}
L^2 = j-2, \;\; (-K_X) \cdot L = j.
\end{equation}
However the latter numerical conditions are not sufficient
to deduce that a linear system consists of rational curves.
Indeed the linear system $3H + E$ on $\Bl_P(\P^2)$, where $H$
is the class of a line on $\P^2$, and $E$ is the
exceptional divisor, satisfies \eqref{eq:num-rat} with $j=10$
but
has a fixed component and contains no rational curves.

We investigate how to find all linear systems
of rational curves of a given degree following
\cite{Manin}.
The first step is to
solve \eqref{eq:num-rat}.  
Let $X$ be a del Pezzo surface of degree $d = K_X^2$.
Assume $X_{\ol{k}}$ is not isomorphic to $\P^1 \times \P^1$
so that $X_{\ol{k}}$ is a blow up of $\P^2$ in $r := 9-d$ points by~\cite[Theorem IV.2.5]{Manin},~\cite[\S 3]{IskovskikhMinModel}.
Write $D = aH - \sum_{i=1}^r b_i E_i$, where $H$ is the pullback
of the hyperplane class and $E_i$ are the 
irreducible components of the exceptional divisor.
Since $-K_X = 3H - \sum_{i=1}^r E_i$,
\eqref{eq:num-rat} translates into:
\begin{equation}\label{eq:manin}
\left\{
\begin{array}{l}
     \sum_{i=1}^{r} b_i = 3a-j  \\
     \sum_{i=0}^{r} b_i^2 = a^2 - j + 2 \\ 
\end{array}
\right.
\end{equation}
The case $j = 1$ is that of lines on del Pezzo surfaces \cite{Manin}.

\begin{lemma}\label{lem:curves}
Let $X$ be a del Pezzo surface of degree $d$ 
over an algebraically closed field $k$.

\begin{enumerate}[(i)]
\item A linear system of rational curves of degree $j$
satisfies $h^0(X, L) = j$, and higher cohomology of $L$ vanish. 
\item Any linear system satisfying \eqref{eq:num-rat}
with $j \ge 1$ is non-empty.
\item A linear system satisfying \eqref{eq:num-rat}
with $1 \le j \le d-1$ is a linear system of rational curves.
\item Let $1 \le j \le d-1$.
The assignment $|D| \mapsto |-K_X -D|$ establishes a bijection
between linear systems of rational curves of degrees $j$
and $d-j$.
\end{enumerate}

\end{lemma}

\begin{proof}
(i) Let $C \in |L|$ be a smooth rational curve.
Taking cohomology for the short
exact sequence 
\[
0 \to \OO_X \to \OO_X(C) \to \OO_C(C) \to 0
\]
we get $H^1(X, \OO_X(C)) = H^2(X,\OO_X(C)) = 0$, and hence $h^0(X,\OO_X(C)) = j$
by Riemann-Roch.

(ii) We 
have $h^2(X, L) = h^0(X, L^\vee \otimes \omega_X) = 0$ (because $L^\vee \otimes \omega_X$ has negative
anticanonical degree) and the Riemann-Roch theorem implies
\[
h^0(X, L) \ge j.
\]

(iii) If $X \simeq \P^1 \times \P^1$,
then it is easy to see that solutions of \eqref{eq:num-rat}
are precisely divisors $(a,1)$, $a \ge 0$ and
$(1,b)$, $b \ge 0$, and these are linear systems
of rational curves (no upper bound on the degree
needed). Now assume that $X$ is a blow up of $\P^2$
in $0 \le r \le 8$ points.
Since $D$ satisfies \eqref{eq:num-rat},
it follows that $-K_X - D$ satisfies \eqref{eq:num-rat}
with degree $d - j$.
By (ii) both $|D|$ and $|-K_X - D|$ are nonempty.
Therefore if $D = aH - \sum_{i=1}^r b_i E_i$, then
$0 \le a \le 3$. 
Solving equations \eqref{eq:manin} with $b_i \in \Z$,
under the assumption $1 \le j \le d-1$
gives rise to the following divisors, which we list
up to reordering of the exceptional divisors:
\begin{itemize}
    \item $a = 0$: $E_1$
    \item $a = 1$: $H - \sum_{i=1}^t E_i$
    \item $a = 2$: $2H - \sum_{i=1}^{r-t} E_i$
    \item $a = 3$: $3H - 2E_1 - \sum_{i=2}^r E_i$
\end{itemize}
Indeed, the cases $a = 0, 1$ are straightforward
and the case $a = 2,3$ follow via the $D \mapsto -K_X - D$
substitution.
The $a = 1, 2$ cases are only possible when $r \le 7$ ($d \ge 2$). We must have $0 \le t \le 2$, and for $r = 6$ (resp. $r = 7$), 
only allow $t = 1, 2$ (resp. $t = 2$).
Under these conditions
each of the divisors in the list is linearly
equivalent to a smooth rational
curve.

(iv) Follows from (iii).
\end{proof}

\begin{proposition}\label{prop:hilbert-scheme}
Let $X/k$ be a del Pezzo surface of degree $d$.
\begin{enumerate}[(i)]
\item For each $j \ge 1$, $\HH^j(X)$ is either
empty or a smooth Severi-Brauer fibration
of relative dimension $j-1$
over
a smooth zero-dimensional scheme $\MM^j(X)$.
For each $1 \le j \le d-1$ we have a natural isomorphism
$\MM^{j}(X) \simeq \MM^{d-j}(X)$.

\item Assume $\HH^j(X)$ is nonempty and let $Z \subset X$ be a 
zero-dimensional subscheme of degree $j-1$.
Let $\HH^j(X,Z)$ be the closed subvariety of $\HH^j(X)$
parametrizing curves containing $Z$. If $\HH^j(X,Z)$ is
nonempty and zero-dimensional,
then it is isomorphic to $\MM^j(X)$.

\item Suppose that 
$X_{\ol{k}}$ is
obtained by blowing up 
$r \le 5$ points on $\P^2$ with exceptional divisors $E_1,\ldots,E_r$. Let $H \in \NS(X_{\ol{k}})$ be the pullback of the hyperplane class.
Then the classes of conics (resp. cubics) on $X_{\ol{k}}$ are $H - E_i$ and $2H - \sum_{i = 1}^4 E_i$ (resp. $H$ and $2H - \sum_{i = 1}^3 E_i$).

\item Consider the following cases.
\begin{enumerate}
    
\item [($dP_8$)] Let $X$ be 
a minimal
del Pezzo surface of degree $8$.
Then $\MM^4(X) \simeq \Spec(k)$ and $\MM^2(X) \simeq \MM^6(X) \simeq
Z_2$,
where $Z_2$ is an étale $k$-scheme of degree 2.

\item [($dP_6$)] Let $X$ be a del Pezzo surface of degree $6$.
Then
\[
\MM^2(X) \simeq \MM^4(X) \simeq
Z_3, \; \MM^3(X) \simeq Z_2
\]
for degree two (resp. degree three) \'etale $k$-schemes $Z_2, Z_3$.

\item [($dP_5$)] Let $X$ be a del Pezzo surface
of degree $5$. 
Then $\MM^2(X) \simeq \MM^3(X) \simeq Z_5$
for a degree five  \'etale $k$-scheme $Z_5$.

\end{enumerate}

\end{enumerate}

\end{proposition}

\begin{proof}
(i) 
Since $H^1(X, \OO_X) = 0$ \cite[III.3.2.1]{Kollarrat},
deforming an effective divisor
$D$ as a subscheme is equivalent to deforming it in
its linear system, so
$\HH^j(X)_{\ol{k}}$
is either empty or a disjoint
union of complete linear systems (see \cite[I.1.14.2]{Kollarrat}). 
Since the Hilbert polynomial of 
a rational curve $C$ on a del Pezzo
surface is determined by its degree $C \cdot (-K_X)$,
and since the Hilbert
scheme is projective \cite[Theorem I.1.4]{Kollarrat},
in particular of finite type, it follows that
$\HH^j(X)_{\ol{k}}$
is a finite disjoint union of projective spaces.
(Finiteness also follows from the fact
there are finitely many solutions for $\eqref{eq:num-rat}$
with fixed $j$.)
By Lemma \ref{lem:curves}(i), these projective spaces
have dimension $j-1$.

Thus we have shown that $\HH^j(X)$ is a smooth
scheme of finite type
over $k$ isomorphic over $\ol{k}$ to a finite
disjoint union of projective spaces.
Let $\MM^j(X) = \Spec(H^0(\HH^j(X), \OO))$; it is a smooth zero-dimensional
scheme, and $\HH^j(X) \to \MM^j(X)$ is a Severi-Brauer fibration.
Finally, the isomorphism $\MM^j(X) \simeq \MM^{d-j}(X)$
is given by $|C| \mapsto |-K_X - C|$, using Lemma \ref{lem:curves} (iv).

(ii) We claim that the projection $\HH^j(X,Z) \to \MM^j(X)$
is an isomorphism. By Galois descent it is sufficient to verify
this over the algebraic closure. 
Each fiber of $\HH^j(X)_{\bar{k}} \to \MM^j(X)_{\bar{k}}$ is equal to $|L|$ for some line bundle $L$. It follows from the assumption that the linear system defined by $H^0(X_{\bar{k}},L \otimes I_Z) \subset H^0(X_{\bar{k}},L)$ is a point in $|L|$ where $I_Z$ be the ideal sheaf of $Z_{\bar{k}} \subset X_{\bar{k}}$. Hence $\HH^j(X,Z)_{\bar{k}} \to \MM^j(X)_{\bar{k}}$
is an isomorphism.

(iii) is proven by solving~\eqref{eq:manin}.

(iv) Galois descent identifies smooth zero-dimensional scheme with
a set with Galois action. Thus it suffices to verify the numbers
of conic bundles and nets of cubics over the algebraic
closure in each case. 
For ($dP_8$), 
the linear system of quartic rational curves on $\P^1 \times \P^1$ is defined by $\OO(1,1)$, and the two pencils of conics $\OO(1,0)$ and $\OO(0,1)$.
The cases ($dP_6$) and ($dP_5$) follow from (iii). 
\end{proof}

\section{Models of large degree}
\label{sec:geom-rat-5}

In this section we deal with surfaces
admitting the most interesting birational automorphisms.
These are
geometrically rational surfaces with minimal 
models of degree $\ge 5$. At the end of this section
we prove
Theorem \ref{thm:main}
and discuss rationality centers for rational surfaces.

\subsection{Virtual N\'eron-Severi sets}

We introduce a type of surfaces 
$X/k$
which we call
\emph{models of large degree}:
\begin{enumerate}
    \item[($dP_9$)] $\P^2$ or a Severi-Brauer surface 
    \item[($dP_8$)] minimal del Pezzo
    surface of degree $8$; in this case $X_{\ol{k}} \simeq \P^1 \times \P^1$    \item[($C_8$)] smooth conic bundle over a conic;
    in this case $X_{\ol{k}} \simeq \F_n$, 
    and we assume $n \ge 1$ to avoid the overlap with ($dP_8$)
    \item[($dP_6$)] del Pezzo surface of degree $6$ 
    \item[($dP_5$)] del Pezzo surface of degree $5$ 
\end{enumerate}

Note that we do not make any assumption
on the Picard rank of the surfaces in the list above.

\begin{proposition}\label{prop:min-rational-dP}
All minimal geometrically
rational surfaces $X$
with $K_X^2 \ge 5$
and all conic bundles over a curve of genus $0$ with $K_X^2 \ge 5$ 
are among models of large degree.
\end{proposition}
\begin{proof}
By the work of Iskovskikh \cite{IskovskikhMinModel}
explained in Theorem \ref{thm:mmp} (3), (4),
$X$ is a del Pezzo surface of Picard rank one
or a conic bundle of Picard rank two.
Minimal del Pezzo surfaces of degree $K_X^2 \ge 5$ are all among
($dP_9$), ($dP_8$), ($dP_6$), ($dP_5$)
(del Pezzo surface of degree $7$ is not minimal,
see e.g.~\cite[Corollary on p.37]{IskovskikhMinModel}).

Assume that $X$ is
a conic bundle over a curve of genus $0$; using \eqref{eq:KX2-conicbundle}
we have
$5 \le K_X^2 \le 8$.
By \cite[Theorem 5]{IskovskikhMinModel},
$X$ is model of large degree.
\end{proof}

\begin{definition}\label{def:MX}
For each model
of large degree
we define its 
\emph{virtual N\'eron-Severi Galois set}
$A_X \in 
\Burn(G_k) \simeq \Z[\Var^0/k]$
as follows.
In each case, $Z_i$ refers to one of the finite \'etale $k$-schemes of degree $i$ introduced in Proposition \ref{prop:hilbert-scheme}(iv).
\begin{itemize}
    \item[($dP_9$)] $A_X = [\Spec(k)]$
    \item[($dP_8$)] $A_X = [Z_2]$, where $Z_2$ parametrizes the rulings (that is, the conic bundle structures)
    on $X_{\ol{k}}$
    \item[($C_8$)] $A_X =  2[\Spec(k)]$ 
    \item[($dP_6$)] $A_X = [Z_2] + [Z_3] - [\Spec(k)]$, where $Z_3$
    parametrizes three pencils of conics on $X_{\ol{k}}$
    and $Z_2$ parametrizes two families of
    cubics

    \item[($dP_5$)] $A_X = [Z_5]$, where $Z_5$ parametrizes five
    pencils of conics on $X_{\ol{k}}$
\end{itemize}
\end{definition}

Note that in all cases except for ($dP_6$), the virtual N\'eron-Severi set
is realized as a set.
The following Lemma explains the name 
N\'eron-Severi set. See \eqref{eq:mu} for the definition of $\mu_{G_k}$.

\begin{lemma}\label{lem:NS}
If $A_X \in \Burn(G_k)$ is the virtual N\'eron-Severi set of $X$,
and $F$ is any field of characteristic zero, then 
$\mu_{G_k}(A_X) = 
[\NS(X_{\ol{k}}) \otimes F] \in K_0(\Rep(G_k, F))$.
\end{lemma}

\begin{proof}
\begin{itemize}
    \item[($dP_9$)] $A_X$ is a set consisting of one element
    and $\NS(X_{\ol{k}})$ is the one-dimensional trivial $G_k$-representation.
    \item[($dP_8$)] $A_X$ is the set of rulings on $X$, and
    $\NS(X_{\ol{k}})$ is freely generated by these rulings.
    \item[($C_8$)] $A_X$ is a set with two elements
    and trivial action.
    Since $K_X$ 
    and the class of a fiber form a basis of
    $\NS(X) \otimes F$ and $\rk \, \NS(X_{\ol{k}}) = \rk \, \NS(\F_n) = 2$, 
    necessarily $\NS(X_{\ol{k}})$
    is of rank two with trivial Galois action.
    \item[($dP_6$)] By Proposition \ref{prop:hilbert-scheme} (iii), $\NS(X_{\ol{k}}) \otimes F$ 
    is generated by three classes of 
    conics and two classes of cubics, 
    modulo a one-dimensional space
    with
    a trivial $G_k$-action;
    this implies the result.
    \item[($dP_5$)]  By Proposition \ref{prop:hilbert-scheme} (iii), $\NS(X_{\ol{k}}) \otimes F$ is
    freely generated by
    classes of conics $A_X$.
\end{itemize}
\end{proof}

It follows from Lemma \ref{lem:NS} that the rank 
of $\NS(X)$ equals the number of orbits of
the virtual N\'eron-Severi set
$A_X$.
In particular, $X$ has Picard rank one if and only
$A_X$ has only one (virtual) orbit.
Since $\mu_{G_k}$
is not injective, we can not simply define
the virtual Galois set of a surface using
$\mu_{G_k}(A_X) = [\NS(X_{\ol{k}}) \otimes F]$,
see Example \ref{ex:ex-cubics} for an illustration; this is why we define them 
case by case in Definition \ref{def:MX}.

\begin{remark}
The Galois sets
forming $A_X$ naturally appear 
in the Chow motive of $X$
\cite[(9) and the following Remark]{Gille}
and as 
semiorthogonal components in the derived category
of the corresponding surface
\cite{BlunkSierraSmith}, 
\cite[Propositions 9.8, 10.1]{AuelBernardara}.
Furthermore, singular versions
of those $k$-algebras show up in the study
of the derived
categories of the corresponding singular
del Pezzo surfaces 
\cite{Kuznetsov-sextics},
\cite{KarmazynKuznetsovShinder},
\cite{Xie-delPezzo}.
\end{remark}

\begin{proposition}\label{prop:geom-rat-5}
Let $X$ be a model of large degree,
and let $\phi: X \dashrightarrow X'$
be a birational isomorphism to another 
minimal surface; then $X'$ is also a model
of large degree, and
\begin{equation}\label{eq:c-MM}
c(\phi) = A_{X'} - A_X.
\end{equation}
\end{proposition}

First we prove that if we apply $\mu_{G_k}$ to both sides of~\eqref{eq:c-MM}, we have equality.

\begin{lem}\label{lem:mu=}
Let $X$, $X'$ be models of large degree.
Let $\phi: X \leftarrow Y \to X'$
be a composition of a blow up and a blow down 
(the centers of the blow ups can be disconnected, or empty).
We have 
$$\mu_{G_k}(A_X') - \mu_{G_k}(A_X) = \mu_{G_k}(c(\phi)).$$
\end{lem}
\begin{proof}
Let $Z$, $Z'$ be the centers of the two blow ups.
By the blow up formula, we have 
$[X] + \L\cdot {[Z]} = [Y] = [X'] + \L\cdot{[Z']}$
in the Grothendieck ring.
Applying the \'etale realization 
\eqref{eq:mu-et}
to this
equality (cf. Example \ref{ex:mu-et-X}),
and using Lemma \ref{lem:NS} with $F = \Q_\ell$
we get
\begin{equation}\label{eq:AXplusZ}
\mu_{G_k}(A_X) +
[\Q_\ell[Z_{\bar{k}}]] = 
[\NS(Y_{\ol{k}}) \otimes \Q_\ell] = 
\mu_{G_k}(A_{X'}) +
[\Q_\ell[Z'_{\bar{k}}]] 
\in K_0(\Rep(G_k,\Q_\ell)).
\end{equation}
Hence $\mu_{G_k}(A_X') - \mu_{G_k}(A_X) = \mu_{G_k}(c(\phi))$.
\end{proof}

Before proving Proposition~\ref{prop:geom-rat-5}, we establish some particular cases,
which rely on Gassmann
equivalence being trivial for
Galois sets of order $\le 5$.

\begin{lem}\label{lem:c-link}
Let $X$, $X'$ be models of large degree.
Let $\phi: X \leftarrow Y \to X'$
be a composition of a blow up and a blow down.
Assume that either
(1) the two blow ups have isomorphic centers
or (2) $K_Y^2 \ge 5$.
Then 
\eqref{eq:c-MM}
holds for $\phi$.
\end{lem}
\begin{proof}

Let $Z$, $Z'$ be the centers of the two blow ups.
We need to show that 
$$\alpha \cnec (A_X + [Z]) - (A_{X'} + [Z']) = 0.$$
By Lemma~\ref{lem:mu=} and
Corollary \ref{cor:Gassmann},
it is sufficient to show that $|\alpha| \le 5$,
that is
to represent $\alpha$
as a difference of two $G_k$-sets, each having order $\le 5$.
Recall that the virtual $G_k$-set
$A_X$ is represented by a $G_k$-set of order
$\le 5$, except
possibly when $K_X^2 = 6$ in which case
$A_X + [\Spec(k)]$ is a $G_k$-set, of order $5$.

(1) Since $[Z] = [Z']$, we have $\alpha = A_X - A_{X'}$.
If $K_X^2 = K_{X'}^2 \ne 6$, then $|\alpha| \le 5$,
as $|A_X| = |A_{X'}| \le 5$.
If $K_X^2 = K_{X'}^2 = 6$, then $\alpha = (A_X + [\Spec(k)]) - 
(A_{X'} + [\Spec(k)])$ also shows that $|\alpha| \le 5$.

(2) Assume both $K_X^2$, $K_{X'}^2$ are not equal to $6$.
In this case by  \eqref{eq:AXplusZ} both $A_X + [Z]$
and $A_{X'} + [Z']$ have order equal to $\rk \; \NS(X_{\ol{k}}) = 10 - K_Y^2 \le 5$, and the original representation of $\alpha$
shows that $|\alpha| \le 5$.
The cases when $K_X^2 = 6$
or $K_{X'}^2 = 6$ are analogous and are left to the reader.
\end{proof}

\begin{proof}[Proof of Proposition \ref{prop:geom-rat-5}]
By \cite[Theorem 2.5]{IskovskikhFact}, any birational isomorphism between
minimal geometrically
rational 
surfaces is a composition of 
Sarkisov links
explained in Definition \ref{def:links}.
Since $c$ is a homomorphism and sends isomorphisms
to zero, it suffices to prove Proposition \ref{prop:geom-rat-5} for every link of type I, II, or III.

For type I links we write $a \leftarrow b$
for a link $\phi: X \leftarrow X'$
with $K_X^2 = a$, $K_{X'}^2 = b$.
We have the following possibilities
according to \cite[Theorem 2.6(i)]{IskovskikhFact}: $9 \leftarrow 8$,
$9 \leftarrow 5$, $8 \leftarrow 6$. 
Here $X' \to B'$ is a conic bundle of
degree $\ge 5$, hence $X'$ is a model of large 
degree by Proposition \ref{prop:min-rational-dP},
and \eqref{eq:c-MM} follows
from Lemma \ref{lem:c-link}(2) (with one
of the centers empty).
Exactly the same argument proves the claim for
links of type III.

For type IIC links, the result is true
by Lemma \ref{lem:cIIC=0} and Proposition \ref{prop:min-rational-dP}.
For a type IID link $X \leftarrow Y \to X'$, by the first statement of 
Proposition \ref{prop:c-nongeomrat}(2),
$K_X^2 \ge 5$
if and only if $K^2_{X'} \ge 5$. 
Hence
$X$ is a model of large degree if and only if $X'$ is.

It remains to show \eqref{eq:c-MM} for each link of 
type IID.
We write $a \leftarrow d \to b$
for a type IID link $X \leftarrow Y \to X'$ between surfaces
of degree $a$, $d$, $b$.
Since $c$ takes values in a torsion-free
abelian group, it vanishes
on involutions and
in particular
the Bertini and Geiser involutions;
these are
links with $d = 1$ and $d = 2$ respectively
in the list
of links in \cite[Theorem 2.6(ii)]{IskovskikhFact}.
On the other hand, links with $d \ge 5$ are covered
by Lemma \ref{lem:c-link}(2).

Thus we only have to consider links with $d = 3$ or $d = 4$.
\begin{claim}\label{claim-deltaIID}
Let $X \xlto{\phi} Y \xto{\phi'} X'$ be a link of type IID such that $K_Y^2 \ge 3$. Let $Z \subset X$   be the blowup center of $\phi$ and $D \subset X$ the divisor contracted by $\phi' \circ \phi^{-1}$. Then each irreducible component of $D_{\bar{k}}$  is a smooth rational curve of degree $\delta$ (according to Definition~\ref{def-deg}) containing exactly $\delta-1$ points of $Z_{\bar{k}}$. For the type IID links with $K_Y^2 \in \{ 3, 4\}$ listed below,  $\gd$ has the following description:
\begin{itemize}
    \item  $9 \lto 4 \to 5, 9 \leftarrow 3 \rightarrow 9$: $\gd = 6$ (which are conics in $\bP^2$ in the classical sense).
    \item $8 \leftarrow 4 \rightarrow 8$: $\gd = 4$.
    \item $8 \lto 3 \to 5$: $\gd = 6$.
    \item $6 \lto 4 \to 6$: $\gd = 3$.
    \item $6 \lto 3 \to 6$: $\gd = 4$.
    \item $5 \lto 4 \to 9$: $\gd = 2$.
    \item $5 \lto 3 \to 8$: $\gd = 3$.
\end{itemize}
\end{claim}
\begin{proof}

Since $K_Y^2 \ge 3$, 
$Y_{\bar{k}}$ is obtained by blowing up 
$r \le 6$ points on $\P^2$ with exceptional divisors $E_1,\ldots,E_r$. Let $H \in \NS(Y_{\bar{k}})$ be the pullback of the hyperplane class of $\P^2$. By solving~\eqref{eq:manin}, the $(-1)$-classes on $Y_{\bar{k}}$ are one of the following:
\begin{itemize}
    \item $E_i$, with $i \in \{1,\ldots,r\}$;
    \item $H - E_i - E_j$, with $i , j \in \{1,\ldots,r\}$ and $i \ne j$;
    \item (only when $r = 6$) $-K_Y - H + E_{i}$, with $i \in \{1,\ldots,6\}$.
\end{itemize}
From the above description, any pair of $(-1)$-curves $E$ and $E'$ on $Y_{\bar{k}}$ satisfies $E \cdot E' \le 1$. Since $X_{\bar{k}}$ is a simultaneous contraction of disjoint $(-1)$-curves on $Y_{\bar{k}}$, it follows that each irreducible component $C$ of $D_{\bar{k}}$ is smooth. It also follows that if $C$ has degree $\delta$, then $C$ contains $\delta-1$ points of $Z_{\bar{k}}$.

The value of $\gd$ follows from the matrix description  in~\cite[Theorem 2.6]{IskovskikhFact}
of the action of each link in the classification on the Picard-Manin space. For instance for $6 \lto 3 \to 6$, the description  in~\cite[Theorem 2.6(ii), $K_X^2 = 6$, (c)]{IskovskikhFact} implies that $D = -2K_X$, which shows that $\gd = (-2K_X^2)/3 = 4$. The same argument works for other links.
\end{proof}

We first consider symmetric links where the centers
of the blow up and the blow down are isomorphic,
so that they are covered by Lemma \ref{lem:c-link}(1):
\begin{itemize}
    \item $9 \leftarrow 3 \rightarrow 9$: 
    the first map blows up a Galois orbit
    of six points, and
    the second one 
    contracts  
    the Galois orbit
    of the proper preimages
    of six conics (in the classical sense: 
    they have degree  6 according to Definition~\ref{def-deg}) passing through five of the points by Claim~\ref{claim-deltaIID};
    these two Galois orbits are isomorphic.
    \item $8 \leftarrow 4 \rightarrow 8$: 
    we blow up
    a Galois orbit of four general points 
    on $X$
    and contract the Galois orbit of the proper preimages of 
    quartic curves
    passing through three of
    the four points by Claim~\ref{claim-deltaIID};
    these two Galois orbits are isomorphic.

\end{itemize}

Finally we need to deal with the following links
(we list them up to inverses):

\begin{itemize}

    \item $8 \leftarrow 3 \rightarrow 5$: here
    $X$ is a quadric with $A_X = [Z_2]$
    and $X'$ is 
    a degree five del Pezzo surface
    with
    $A_{X'} = [Z_5']$. 
    We need to show that the first map
    has center $Z_5'$ and the second map has center $Z_2$.
    By Claim~\ref{claim-deltaIID},
    the center of the second map parametrizes 
    smooth rational curves of degree $6$ on $X$ 
    passing through an orbit of five points.
     By Proposition
    \ref{prop:hilbert-scheme}(ii) and (iv)$(dP_8)$,
    the center of the second map
    is a subscheme isomorphic to $Z_2$.
    The center of the first map parametrizes cubics
    passing through an orbit of two points on $X'$, and thus 
    by Proposition \ref{prop:hilbert-scheme}(ii) and (iv)$(dP_5)$ 
    is a scheme
    isomorphic to $Z_5'$.

    \item $9 \leftarrow 4 \rightarrow 5$: it suffices to show that
    if $Z_5$ is the center of the first map and $A_{X'} = [Z_5']$,
    then $Z_5 \simeq Z_5'$.
    One can verify this directly, as in the previous link, or apply 
    Corollary \ref{cor:L-equiv} as in the proof of Lemma \ref{lem:c-link}(2).

    \item $6 \leftarrow 3 \rightarrow 6$: we have
    $A_X = [Z_3] + [Z_2] - 1$
    and $A_{X'} =  [Z_3'] + [Z_2'] - 1$.
    We claim that $c(\phi) = [Z_3'] - [Z_3]$
    and $Z_2 \simeq Z_2'$.
    Indeed, by Claim~\ref{claim-deltaIID}
    the second arrow
    contracts the proper preimages of smooth 
    rational curves
    of degree $4$ passing through
    the Galois orbit of the points blown up by the second arrow,
    and the latter scheme of conics is isomorphic
    to $Z_3$ by Proposition \ref{prop:hilbert-scheme}(ii) and (iv)$(dP_6)$.
    The same argument with roles of $X$, $X'$ reversed
    implies that the first arrow blows up $Z_3'$.
    Thus $c(\phi) = [Z_3'] - [Z_3]$.
    By Lemma~\ref{lem:mu=}, we have $\mu_{G_k}([Z_2]) = \mu_{G_k}([Z'_2])$, hence
    $Z_2 \simeq Z_2'$
    by Proposition
    \ref{prop:Gassmann}. 
    \item $6 \leftarrow 4 \rightarrow 6$: we have
    $Z_3 \simeq Z_3'$, and $c(\phi) = [Z_2'] - [Z_2]$
    similarly to the previous case,
    using Proposition \ref{prop:hilbert-scheme}(ii), (iv)$(dP_6)$,
    and Proposition
    \ref{prop:Gassmann} again.
\end{itemize}
\end{proof}

\begin{proof}[Proof of Theorem \ref{thm:main}]
We first assume that $X$ and $Y$
are geometrically irreducible. Composing with contractions to minimal models,
and using the additivity of $c$ under composition,
we may assume that $X$ and $Y$ are minimal,
hence belong to one of the classes from Theorem
\ref{thm:mmp}.

In the nongeometrically rational case
and geometrically rational case with $K_X^2 \le 4$
the result follows from Proposition \ref{prop:c-nongeomrat}.
Finally, in 
the geometrically rational case with $K_X^2 \ge 5$
the result is Proposition \ref{prop:geom-rat-5}.

In general, that is 
if the surfaces $X$ and $Y$ are not geometrically
irreducible, write $L_X$ and $L_Y$
for the fields of regular functions
of $X$ and $Y$;
these are finite field extensions of $k$.
Then $\phi$ induces a $k$-isomorphism $\sigma: L_X \to L_Y$,
which allows us to consider both $X$ and $Y$ as smooth projective
geometrically irreducible surfaces over $L := L_X$
and $\phi$ becomes a $L$-birational isomorphism.
This way $X$ and $Y$ are restrictions of scalars of geometrically
irreducible surfaces over $L$ (this can be
thought of as the Stein factorizations for
$X$ and $Y$ over $\Spec(k)$), and the result follows
from the geometrically irreducible case
considered above and Proposition \ref{prop:c-Galois}(ii).
\end{proof}

\subsection{Rationality centers}
\label{subsec:rat-cent}
The following
corollary tells us that the rationality center of a
rational surface
$X$ is well-defined, that is
for any sequence of blow ups
and blow downs connecting $\P^2$
to $X$, 
the virtual Galois set of 
blow up centers minus the blow down centers
is independent
of the 
choice of the
birational isomorphism between them.

In the higher-dimensional case such rationality
centers are not well-defined (however, see
Definition 2.3 in \cite{GalkinShinder}
for a similar class in the
localized Grothendieck ring).

\begin{corollary}[of Theorem \ref{thm:main}]
\label{cor:rat-centers}
There exists a unique map
\[
\{ \text{Isomorphism classes of rational smooth projective surfaces} \} \overset{\MM}{\longrightarrow} \Z[\Var^0/k]
\]
with the following properties:
\begin{itemize}
    \item[(1)] We have $\MM(X) = A_X$
    as in Definition \ref{def:MX}
for models of large degree 

\item[(2)] For any birational isomorphism
$\phi: \P^2 \dashrightarrow X$
we have $c(\phi) = [\MM(X)] - 1$
\end{itemize}

\end{corollary}
\begin{proof}

For any rational surface
we define $\MM(X)$ as $c(\phi) + 1$; by Theorem \ref{thm:main}
this is independent of the choice of $\phi: \P^2 \dashrightarrow X$. As 
$\cM$ is required to satisfy $[\MM(X)] = c(\phi) + 1$, this shows the uniqueness of $\cM$.
By Proposition \ref{prop:geom-rat-5} this is consistent
with Definition \ref{def:MX}.
\end{proof}

\begin{example}\label{ex:dP5-center}
If $X$ is a del Pezzo surface
of degree $5$, by Definition \ref{def:MX},
$A_X = [Z_5]$. 
Assume that $X$ has Picard rank one which
by Lemma \ref{lem:NS}
is equivalent to $Z_5$ being irreducible.
Then for any rationality
construction of $X$ (see Example \ref{ex:dP5-links}), that is a sequence of
blow ups and blow downs starting with $\P^2$
and ending with $X$, one of the blow ups will
have $Z_5$ as its center.
\end{example}

\begin{example}
Consider the rationality question for a del Pezzo surface $X$ of degree $d \le 4$. If $X$ is rational, then
it is a result of Iskovskikh 
deduced from his classification 
of links
\cite[Theorem 2.6]{IskovskikhFact} 
that such 
a surface $X$ can not be minimal,
thus it admits a Galois orbit of disjoint $(-1)$-curves which
we can contract via some morphism
$\phi: X \to X'$ to obtain a minimal rational
del Pezzo surface
of degree $5$, $6$, $8$ or $9$.
In each case the rationality center $A_X$
is given in the table.
As usual we write $Z_j$ or $Z_j'$ for \'etale schemes
of degree $j$.
\medskip
\begin{center}
\begin{tabular}{|c|c|c|c|}
     \hline
     $\deg(X')$ & $c(\phi^{-1})$ & $A_{X'}$ & $A_X$ \\
     \hline
     $5$ & $[Z_{5-d}]$ & $[Z_5']$ & $ [Z_{5-d}] + [Z_5'] $\\
     \hline
     $6$ & $[Z_{6-d}]$ & $[Z_2']+[Z_3']-1$ & $[Z_{6-d}] + [Z_3'] + [Z_2'] - 1$\\
     \hline
     $8$ & $[Z_{8-d}]$ & $[Z_2']$ & $[Z_{8-d}] + [Z_2']$\\
     \hline
     $9$ & $[Z_{9-d}]$ & $1$ & $[Z_{9-d}] + 1$\\
     \hline
\end{tabular}
\end{center}
\end{example}

\begin{example}\label{ex:cubics-M}
Consider the two rational cubic surfaces $X$, $X'$
introduced in Example \ref{ex:ex-cubics}; they have isomorphic
rational permutation N\'eron-Severi representations, 
and the associated Galois sets can not be read off from them.

However, the construction of $\MM_X$ does determine these Galois sets from $X$ itself,
as by Corollary \ref{cor:rat-centers}
we have 
$[Z] = \MM_X - 1$ and $[Z'] = \MM_{X'} - 1$.
\end{example}

\bibliographystyle{alpha}
\bibliography{Cr2Leq}

\end{document}